\documentclass[12pt]{article}
\usepackage[latin1]{inputenc}

\usepackage{enumitem}

\usepackage{amsmath,amsthm,amssymb,amsfonts}
\usepackage{mathtools}
\usepackage[pdftex]{graphicx}
\usepackage{graphicx}
\usepackage{epsfig}
\usepackage{verbatim}
\usepackage{footmisc}
\usepackage{bbm,bm}
\usepackage{lscape}
\usepackage[longnamesfirst]{natbib}
\usepackage{color}
\usepackage{float}
\usepackage{longtable}
\usepackage{lipsum}

\usepackage{booktabs}

\usepackage{natbib}

\newtheorem{Theorem}{Theorem}[section] 
\newtheorem{Lemma}[Theorem]{Lemma} 
\newtheorem{Remark}[Theorem]{Remark} 
\newtheorem{Corollary}[Theorem]{Corollary} 
 
\newtheorem{Definition}[Theorem]{Definition}

\makeatletter

\@addtoreset{equation}{section}
\makeatother

\renewcommand{\baselinestretch}{1.2}

\newcommand{\R}{\mathbb{R}}
\newcommand{\N}{\mathbb{N}}
\newcommand{\Z}{\mathbb{Z}}

\newcommand{\cF}{\mathcal{F}}
\newcommand{\cB}{\mathcal{B}}
\newcommand{\cC}{\mathcal{C}}
\newcommand{\cD}{\mathcal{D}}
\newcommand{\cH}{\mathcal{H}}

\newcommand{\cX}{\mathcal{X}}
\newcommand{\cA}{\mathcal{A}}

\newcommand{\Cov}{\operatorname{Cov}}

\DeclareMathOperator*{\cart}{\times}

\newcommand{\si}{\sigma}
\newcommand{\ep}{\varepsilon}
\newcommand{\om}{\omega}

\newcommand{\lf}{\left\lfloor}
\newcommand{\rf}{\right\rfloor}
\newcommand{\mi}{\wedge}

\newcommand{\ph}{\varphi}

\allowdisplaybreaks

\usepackage{color}

\usepackage{hyperref}

\begin{document}

\begin{titlepage}
\title{\bf Consistent nonparametric change point detection combining CUSUM and marked empirical processes}

\author{{\sc Maria Mohr} and {\sc Natalie Neumeyer}\\ Department of Mathematics, University of Hamburg}

\maketitle

\renewcommand{\baselinestretch}{1.1}

\begin{abstract}
A weakly dependent time series regression model with multivariate covariates and univariate observations is considered, for which we develop a procedure to detect whether the nonparametric conditional mean function is stable in time against change point alternatives. Our proposal is based on a modified CUSUM type test procedure, which uses a sequential marked empirical process of residuals. We show weak convergence of the considered process to a centered Gaussian process under the null hypothesis of no change in the mean function and a stationarity assumption. This requires some sophisticated arguments for sequential empirical processes of weakly dependent variables. As a consequence we obtain convergence of Kolmogorov-Smirnov and Cramér-von Mises type test statistics. The proposed procedure acquires a very simple limiting distribution and nice consistency properties, features from which related tests are lacking.
We moreover suggest a bootstrap version of the procedure and discuss its applicability in the case of unstable variances. 
\end{abstract}

\vspace*{.5cm}



\end{titlepage}

\small
\normalsize
\addtocounter{page}{1}

\section{Introduction}  \label{introduction}

Assume a finite sequence $(\bm{X}_t,Y_t)$, $t=1,\dots,n$, of a weakly dependent $\mathbb{R}^d\times\mathbb{R}$-valued time series 
 has been observed. Here, we interpret $\bm{X}_t$ as a covariate (which may contain past values of the process) and it is assumed that the conditional expectation of the observation $Y_t$, given $\bm{X}_t$ and all past values of the time series, does only depend on the covariate $\bm{X}_t$, and thus is a function $m_t(\bm{X}_t)$. We do not impose any parametric structure on the regression function.  For inference on the time series it is of importance whether the regression function is time dependent or not, i.e.\ the hypothesis 
\[H_0: \ m_t(\bm{X}_t)=m(\bm{X}_t) \text{ a.s. for all }t=1,\dots,n\]
(for some not further specified function $m$) should be tested against structural changes over time such as change point alternatives. 

Literature on such tests for nonparametric regression functions is rare in the time series context. 
Both \cite{Hidalgo1995671} and \cite{Honda199745} suggested CUSUM tests for change points in the regression function in nonparametric time series regression models with strictly stationary and absolutely regular data. \cite{Su2008347} extended these tests to  strongly mixing and not necessarily stationary  processes, allowing for heteroscedasticity, 
while \citet{Su20101761} proposed change point tests in partially linear time series models. \citet{Vogt2015811} constructed a kernel-based $L_2$-test for structural change in the regression function in time-varying nonparametric regression models with locally stationary regressors. 

We will combine the CUSUM approach as considered by \cite{Hidalgo1995671}, \cite{Honda199745}, and \cite{Su2008347} with a marked empirical process approach. 
Marked empirical processes have been suggested in a seminal paper by \cite{Stute1997613} for lack-of-fit testing in nonparametric regression models with i.i.d.~data. Since then they have been widely used for hypothesis testing in regression models, see \cite{Koul1999204} and \cite{Delgado20011469}, among many others. 
A marked empirical process approach has been applied by \cite{Burke2013261}  for change point detection in an i.i.d.~setting. In contrast to our approach they use a process of observations instead of residuals with a very complicated limit distribution, whereas we obtain a  simple limit distribution and even asymptotically distribution-free tests in the case of one-dimensional covariates. 
To this end we show weak convergence of a sequential marked empirical process of residuals under the null hypothesis. We further demonstrate consistency under fixed alternatives of one structural break in the regression function at some time point $\lfloor ns_0\rfloor$ for $n\to\infty$. 

Moreover we suggest a wild bootstrap version of our test that can be applied to detect changes in the mean function in the case of stable variances (as alternative to using the asymptotic distribution, e.g.\ for multivariate covariates) as well as in the case of non-stable variances. 
Wild bootstrap was first introduced by \cite{Wu19861261} and \cite{Liu19881696} for linear regression with heteroscedasticity. It was used in time series context by \cite{Kreiss1997} and \cite{Hafner2000177}, among others. 
The bootstrap version of our test can detect changes in the conditional mean function, even when the conditional variance function is also not stable, but -- as desired -- the test does not react sensitive to the unstable variance. If no change in the mean function is detected, a test for change in the variance function can be applied, which assumes a stable mean function. The latter approach will be considered in detail in a forthcoming manuscript.  
Most literature assumes stationary variances of the error terms (unconditional or conditional) when testing for changes in regression.
However, as \cite{Wu2016151} pointed out, non-stationary variances can occur and will most likely result in misleading inferences when not taken into account. Although this is a legitimate concern, not many results are available that deal with non-stationary variances. The CUSUM test by \cite{Su2008347} allows for breaks in the conditional variance function. But their procedure does only seem to work for fixed breaks that do not depend on the sample size, whereas we consider changes of the variance in some $\lfloor nt_0\rfloor$ for $n\to\infty$. There are some approaches for testing for parameter stability in parametric time series models that consider unstable variances, see \cite{Pitarakis200432}, \cite{Perron2008}, \cite{Kristensen2012420}, \cite{Cai2007163}, \cite{Xu2015274} and \cite{Wu2016151}. But all the settings considered do not fit into our framework as they either do not allow for autoregression models, by assuming stationarity of the regressor variables under the null, or they do not cover heteroscedastic effects. More precisely if heteroscedasticity is considered, variance instabilities are not modeled in the conditional variance function but as a time-varying constant.  

The paper is organized as follows. In section 2 we present the model and the sequential marked empirical process, on which the test statistics are based. Further the assumptions are listed. In section 3 we consider the limit distribution under the null hypothesis as well as consistency under the fixed alternative of one change point. The wild bootstrap version of the procedure is discussed in section 4, whereas simulations and a real data example are presented in section 5. Section 6 contains concluding remarks, whereas proofs are presented in the appendix. Some technical details and additional simulation results are deferred to a supplement.

\section{The model and test statistic} \label{sec-model}

Let $(Y_t,\bm{X}_t)_{t\in\Z}$ be a weakly dependent stochastic process in $\R\times\R^d$ following the regression model
\begin{equation*} 
Y_t=m_t(\bm{X}_t)+U_t,\quad t\in\Z.
\end{equation*}
The covariate $\bm{X}_t$ may include finitely many lagged values of $Y_t$, for instance $\bm{X}_t=(Y_{t-1},\dots,Y_{t-d})$ such that the model includes nonparametric autoregression. 
The unobservable innovations $(U_t)_{t\in\Z}$ are assumed to fulfill $E[U_t|\cF^{t}]=0$ almost surely for the sigma-field $\cF^{t}=\sigma(U_{j-1},\bm{X}_{j}:j\le t)$. 
Our assumptions on the innovations are very weak; in particular heteroscedastic models will be covered.
Assuming $(Y_1,\bm{X}_1),\dots,(Y_n,\bm{X}_n)$ have been observed, our aim is to test the null hypothesis 
\begin{align*} 
H_0: m_t(\cdot)=m(\cdot),  \ t=1,\dots,n,
\end{align*}
for the conditional mean function $E[Y_t|\bm{X}_t=\bm{x}]=m_t(\bm{x})$, $t\in\Z$, and some not specified function $m:\R^d\to\R$ not depending on the time of observation $t$. 
%
%
To test $H_0$, we define the  \textit{sequential marked empirical process of residuals} as
\begin{equation}\label{hat-Tn}
\hat{T}_n(s,\bm{z})=\frac{1}{\sqrt{n}}\sum\limits_{i=1}^{\lf ns \rf} (Y_i-\hat{m}_n(\bm{X}_i))\om_n(\bm{X}_i)I\{\bm{X}_i\le \bm{z}\},
\end{equation}
for $s\in[0,1]$ and $\bm{z}\in\R^d$, where $\om_n(\cdot)=I\{\cdot\in\bm{J}_n\}$ with $\bm{J}_n$ from assumption \textbf{(J)} below. Throughout $\bm{x}\le \bm{y}$ is short for $x_j\le y_j, \ \forall \ j=1,\dots,d$, and we use the notations $\lf x \rf=\max \{k\in\Z:k\le x\}$ for $x\in\R$ and  $\bm{x}\mi\bm{y}=(\min\{x_1, y_1\},\dots,\min\{x_d, y_d\})$ as well as $\int _{(-\bm{\infty},\bm{x}]}g(\bm{u})d\bm{u}=\int_{-\infty}^{x_d}\dots\int_{-\infty}^{x_1}g(u_1,\dots,u_d)du_1\dots du_d$; further $I\{\cdot\}$ denotes the indicator function. 

The regression function $m$ is estimated by the Nadaraya-Watson estimator $\hat{m}_n$, where
\begin{eqnarray}\label{NadWat}
\hat{m}_n(\bm{x})&=& \dfrac{\sum_{j=1}^{n}K\Big(\frac{\bm{x}-\bm{X}_{j}}{h_n}\Big)Y_{j}}{\sum_{j=1}^{n}K\Big(\frac{\bm{x}-\bm{X}_{j}}{h_n}\Big)}\label{eq:NadWat}
\end{eqnarray}
with kernel function $K$ and bandwidth $h_n$ as considered in the assumptions below.

The proposed test is a modification of the CUSUM test in \cite{Su2008347}. They consider the process 
\[\frac{1}{\sqrt{n}}\sum\limits_{i=1}^{\lf ns \rf} (Y_i-\hat{m}_n(\bm{X}_i))\hat{f}_n(\bm{X}_i)w(\bm{X}_i),\]
where $w:\R^d\to\R$ is a weighting function and $\hat{f}_n$ is the kernel density estimator. While the factor $\hat{f}_n$ has technical reasons as small random values in the denominator of $\hat{m}_n$ can be avoided, the weighting function $w$ plays a crucial role for the power of their test  (see remarks to Theorem 3.2 in \cite{Su2008347}). Depending on the alternative, $w$ needs to be chosen appropriately, while their test (in contrast to the one based on the sequential marked process) is not generally consistent.

Under the null hypothesis $H_0$ we formulate the following assumptions in order to derive the limiting distribution of $\hat T_n$ and corresponding test statistics in the next section. 

\begin{enumerate}
\item[\textbf{(G)}] Let $(Y_t,\bm{X}_t)_{t\in\Z}$ be strictly stationary and $\alpha$-mixing with mixing coefficient $\alpha(\cdot)$ such that $\alpha(t)=O(a^{-t})$
for some $a\in(1,\infty)$. 
%
\item[\textbf{(U)}] For some $\gamma>0$ and some even $Q>(d+1)(2+\gamma)$, and $\cF^t=\sigma(U_{j-1},\bm{X}_j:j\le t)$, let $E[U_t|\cF^t]=0$,  $E[U_t^2|\bm{X}_t]=\si^2(\bm{X}_t)$ and
 $E[|U_t|^{Q\frac{2+\gamma}{2}}|\bm{X}_t]\le c(\bm{X}_t)^Q$ a.s. for all $t\in\Z$, for some functions $c, \sigma^2: \R^d\to\R$ with 
$\int \bar{c}(\bm{u})dF(\bm{u})\le M$ for some $M<\infty$ and $\bar{c}(\bm{u})=\max\left\{\si^2(\bm{u}),c(\bm{u})^2,\dots,c(\bm{u})^Q\right\}$. 
%
\item[\textbf{(M)}] For some $b>2$ let $E[|Y_1|^b]<\infty$ and let  $\bm{X}_1$ be absolutely continuous with density function $f:\R^d\to\R$ that satisfies $\sup_{\bm{x}\in\R^d}E[|Y_1|^b|\bm{X}_0=\bm{x}]f(\bm{x})<\infty$ and $\sup_{\bm{x}\in\R^d}f(\bm{x})<\infty$. Let there exist some $ j^*<\infty$ such that  $\sup_{\bm{x}_1,\bm{x}_j}E[|Y_1Y_j||\bm{X}_1=\bm{x}_1,\bm{X}_j=\bm{x}_j]f_{1j}(\bm{x}_1,\bm{x}_j)<\infty$ for all $ j\ge j^*$, where $f_{1j}$ is the density function of $(\bm{X}_1,\bm{X}_j)$.
\label{page:(J)}
\item[\textbf{(J)}] Let $(c_n)_{n\in\N}$ be a positive sequence of real valued numbers satisfying $c_n\to \infty$ and $c_n=O((\log{n})^{1/d})$ and let $\bm{J}_n=[-c_n,c_n]^d$.
\label{page:(F)}
\item[\textbf{(F1)}] For some $C<\infty$ and $c_n$ from assumption \textbf{(J)} let $\bm{I}_n=[-c_n-Ch_n,c_n+Ch_n]^d$ and let $\delta_n^{-1}=\inf_{\bm{x}\in \bm{J}_n}f(\bm{x})>0$ for all $n\in\N$. Further, let for some $r,l\in\N$ and for all $n\in\N$
\begin{eqnarray*}
p_n&=&\max\limits_{\substack{\bm{k}\in\N_0^d\\1\le |\bm{k}|\le l+1+r}}\sup\limits_{\bm{x}\in \bm{I}_n}|D^{\bm{k}}f(\bm{x})|<\infty\\
0<q_n&=&\max\limits_{\substack{\bm{k}\in\N_0^d\\0\le |\bm{k}|\le l+1+r}}\sup\limits_{\bm{x}\in \bm{I}_n}|D^{\bm{k}}m(\bm{x})|<\infty,
\end{eqnarray*}
where $|\bm{i}|=\sum_{j=1}^{d}i_j$ and $D^{\bm{i}}=\frac{\partial^{|\bm{i}|}}{\partial x_1^{i_1}\dots\partial x_d^{i_d}}$ for $\bm{i}=(i_1,\dots,i_d)\in\N_0^d$.
\item[\textbf{(F2)}] For $q_n$ from assumption \textbf{(F1)}, $c_n$ from assumption \textbf{(J)} and $C$ from assumption \textbf{(K)}, let for all $\bm{k}\in\N_0^d$ with $|\bm{k}|=2$, $\sup_{\bm{x}\in[-c_n-2h_nC,c_n+2h_nC]^d}\left|D^{\bm{k}}m(\bm{x})\right|=O(q_n)$.
\item[\textbf{(K)}] Let $K:\R^d\to\R$ be symmetric in each component, $l+1$ times differentiable with $\int_{\R^d}K(\bm{z})d\bm{z}=1$ and compact support $[-C,C]^d$. Additionally,  let $r\ge 2$ and
$\int_{\R^d}K(\bm{z})\bm{z}^{\bm{k}}d\bm{z}=0$  for all $\bm{k}\in\N_0^d$  with $1\le |\bm{k}|\le r-1$,  where $\bm{z}^{\bm{k}}=z_1^{k_1}\cdots z_d^{k_d}$.
For all $L\in\{K\}\cup\{D^{\bm{k}}K: \bm{k}\in\N_0^d \text{ with } 1\le |\bm{k}|\le l+1\}$ let  $|L(\bm{u})|<\infty$ for all $\bm{u}\in\R^d$ and $|L(\bm{u})-L(\bm{u'})|\le \Lambda \|\bm{u}-\bm{u'}\|$ for some $\Lambda<\infty$ and for all $\bm{u},\bm{u'}\in\R^d$. (Here, $r,l$ and $C$ are from assumption \textbf{(F1)}.)
\item[\textbf{(B1)}]
For $\delta_n,p_n,q_n$ and $r,l$ from assumption \textbf{(F1)} let 
\begin{align*}
\left(\sqrt{\frac{\log n}{nh_n^{d+2(l+1)}}}+h_n^{r}p_n\right)p_n^{l+1}\delta_n^{l+2}=O(1),
\end{align*}
and for some $\eta\in(0,1)$ let
\begin{align*}
\left(\sqrt{\frac{\log n}{nh_n^{d+2(l+1)}}}+h_n^{r}p_n\right)p_n^{l+\eta}q_n\delta_n^{l+1+\eta}=o(1).
\end{align*}
\item[\textbf{(B2)}] For $l, p_n, q_n, \delta_n$ from assumption \textbf{(F1)} and $\eta$ from assumption \textbf{(B1)}, let $h_n$ satisfy the following conditions
$$\frac{(\log n)^{3+\frac{d}{l+\eta}}}{\sqrt{n^{1-\frac{d}{l+\eta}}h_n^d}}q_n^2\delta_n^2=o(1), \;
\frac{\log {h_n}}{\sqrt{nh_n^d}}=o(1) ,\;
\sqrt{n}h_n^{r}p_nq_n=o(1),\;
(\log n)^3h_nq_n^2=o(1).$$
\end{enumerate}

%
%

\begin{Remark}
Under aforementioned assumptions, consistency properties hold for $\hat{m}_n$ uniformly on $\bm{J}_n$ from assumption \textbf{(J)} which will be shown in section \ref{Appendix-nonparametric} of the appendix. The key tool here is an application of Theorem 2 in \cite{Hansen2008726}. Assumption \textbf{(G)} implies polynomial mixing rates of the underlying process needed in \cite{Hansen2008726}. Moreover, together with the first bandwidth condition in \textbf{(B2)} the bandwidth constraints in \cite{Hansen2008726} are also fulfilled. Assumptions \textbf{(M)} and parts of \textbf{(K)} are reproduced from aforementioned paper. 

In order to satisfy the first bandwidth condition in \textbf{(B2)}, a necessary condition on the smoothness of $f$ and $m$ then is $l+\eta>d$, meaning that for higher dimensional covariate $\bm{X}_t$, the existence of higher order partial derivatives of $f$ and $m$ is needed. 
In order to satisfy both the first and third bandwidth condition in \textbf{(B2)}  at the same time, the order of the kernel needs to be large, 
in particular $r>\frac{d}{2}\frac{l+\eta}{l+\eta-d}$.
The second bandwidth condition in \textbf{(B2)} is implied by the first one, if the bandwidth $h_n$ has a polynomial rate of decay in $n$ (or slower), meaning if there exists a $k\in (0,\infty)$ such that $h_n=O(n^{-k})$. Note that $k<\frac{1}{d}-\frac{1}{l+\eta}$ is necessary then.
\end{Remark}


\section{Asymptotic results} \label{asympt}

To derive the asymptotic distribution of test statistics built from  the sequential marked empirical process $\hat T_n$ defined in (\ref{hat-Tn}), we apply the following expansion, which uses $Y_i=m(\bm{X}_i)+U_i$ for all $i=1,\dots,n$ under the null hypothesis, 
\begin{eqnarray*}
\hat T_n(s,\bm{z}) &=& A_{n2}(s,\bm{z})+A_{n1}(s,\bm{z})
\end{eqnarray*}
with 
\begin{align}\label{A_n1}
A_{n1}(s,\bm{z}):=&\frac{1}{\sqrt{n}}\sum\limits_{i=1}^{\lf ns \rf} (m(\bm{X}_i)-\hat{m}_n(\bm{X}_i))\om_n(\bm{X}_i)I\{\bm{X}_i\le \bm{z}\}\\
A_{n2}(s,\bm{z}):=&\frac{1}{\sqrt{n}}\sum\limits_{i=1}^{\lf ns \rf}U_i\om_n(\bm{X}_i)I\{\bm{X}_i\le \bm{z}\}.\label{A_n2}
\end{align}
Lemma \ref{Lemma1} in the appendix shows that $A_{n2}(s,\bm{z})=T_n(s,\bm{z})+o_P(1)$ uniformly in $s\in [0,1]$ and $\bm{z}\in\R^d$ with the process
\begin{equation}\label{Tn}
T_n(s,\bm{z})=\frac{1}{\sqrt{n}}\sum\limits_{i=1}^{\lf ns \rf}U_iI\{\bm{X}_i\le \bm{z}\},\quad s\in[0,1], \bm{z}\in\R^d.
\end{equation}
Further, Lemma \ref{Lemma2} in the appendix shows that 
\begin{align*}
A_{n1}(s,\bm{z})
&=s\sqrt{n}\int_{\R^d}(m(\bm{x})-\hat{m}_n(\bm{x}))\om_n(\bm{x})I\{\bm{x}\le \bm{z}\}f(\bm{x})d\bm{x}+o_P(1)
\end{align*}
holds uniformly in $s\in [0,1]$ and $\bm{z}\in\R^d$. Inserting the definition of $\hat m_n$ from (\ref{NadWat}) one obtains one term of the form 
$$\frac{s}{\sqrt{n}}\sum\limits_{i=1}^{n}\int_{(-\bm{\infty},\bm{z}]}\left(m(\bm{y})-m(\bm{X}_i)\right)K_{h_n}(\bm{y}-\bm{X}_i)\om_n(\bm{y})\frac{f(\bm{y})}{\hat{f}_n(\bm{y})}d\bm{y},$$
which is negligible by Lemma \ref{Lemma:Teil mit m} and one term of the form 
$$-\frac{s}{\sqrt{n}}\sum\limits_{i=1}^{n}U_i\int_{(-\bm{\infty},\bm{z}]}K_{h_n}(\bm{y}-\bm{X}_i)\om_n(\bm{y})\frac{f(\bm{y})}{\hat{f}_n(\bm{y})}d\bm{y}$$
which can further be expanded applying Lemmata \ref{Lemma:Teil mit U} and \ref{Lemma1} such that one obtains 
$$A_{n1}(s,\bm{z}) =-sT_n(1,\bm{z})+o_P(1)$$
uniformly in $s\in[0,1]$ and $\bm{z}\in\R^d$. From this the expansion given in the first part of Theorem \ref{decomposition} below follows. In the second part of the theorem weak convergence of $T_n$ from (\ref{Tn}) is stated. 

\begin{Theorem} \label{decomposition}
{\bf (i)} Suppose that \textbf{(G)}, \textbf{(U)}, \textbf{(M)}, \textbf{(J)}, \textbf{(F1)}, \textbf{(F2)}, \textbf{(K)}, \textbf{(B1)} and \textbf{(B2)} are satisfied. Then under $H_0$ 
\begin{align*}
\hat{T}_n(s,\bm{z})=T_n(s,\bm{z})-sT_n(1,\bm{z})+o_P(1),
\end{align*}
holds uniformly in $s\in[0,1]$ and $\bm{z}\in\R^d$.

{\bf (ii)}
Suppose that the assumptions \textbf{(G)} and \textbf{(U)} are satisfied. Then under $H_0$ the process $T_n$ converges weakly in $\ell^{\infty}([0,1]\times\R^d)$ to a centered Gaussian process $G$ with
\begin{align*}
\Cov\big(G(s_1,\bm{z}_1),G(s_2,\bm{z}_2)\big)=(s_1 \mi s_2)\Sigma(\bm{z}_1 \mi \bm{z}_2)
\end{align*}
and $\Sigma:\R^d\to\R, \bm{x}\mapsto\int_{(-\bm{\infty},\bm{x}]}\si^2(\bm{u})f(\bm{u})d\bm{u}$.
\end{Theorem}

The proof of the first part follows from the considerations above applying Lemmata \ref{Lemma2}--\ref{Lemma:Teil mit U} in the appendix, while the proof of the second part is given in section \ref{Appendix-proof} of the appendix. The proof of the second part  in particular makes use of a recent result on weak convergence of sequential empirical processes indexed in function classes that can be applied for strongly mixing sequences, see \cite{Mohr2017}. Note that \cite{Koul1999204} show a weak convergence result applicable to the non-sequential process $\{T_n(1,z):z\in\R\}$ under less restrictive assumptions on the dependence structure and moments (see Lemma 3.1 in aforementioned reference). From Theorem \ref{decomposition} and the continuous mapping theorem one directly obtains the limit distribution of $\hat T_n$.

\begin{Corollary}\label{cor:T_dach}
Suppose that the assumptions of Theorem \ref{decomposition}(i) are satisfied. Then under $H_0$ the process $\hat T_n$ converges weakly in 
 $\ell^{\infty}([0,1]\times\R^d)$ to a centered Gaussian process $G_0$ with
\begin{align*}
\Cov\big(G_0(s_1,\bm{z}_1),G_0(s_2,\bm{z}_2)\big)=(s_1 \mi s_2-s_1s_2)\Sigma(\bm{z}_1 \mi \bm{z}_2)
\end{align*}
and $\Sigma$ as in Theorem \ref{decomposition}(ii). 
\end{Corollary}

Continuous functionals of the process $\hat{T}_n$ can be used as test statistics for $H_0$. We consider the following Kolmogorov-Smirnov and Cram\'er-von Mises type statistics and combinations of both, 
\begin{eqnarray*}
T_{n1}&=& \sup\limits_{s\in[0,1],\bm{z}\in\R^d}\left|\hat{T}_n(s,\bm{z})\right|,\quad
T_{n2}\;=\;\sup\limits_{\bm{z}\in\R^d}\int_{0}^1\left|\hat{T}_n(s,\bm{z})\right|^2ds,\\
T_{n3}&=&\sup\limits_{s\in[0,1]}\int_{\R^d}\left|\hat{T}_n(s,\bm{z})\right|^2v(\bm{z})d\bm{z},\quad 
T_{n4}\;=\;\int_{0}^1\int_{\R^d}\left|\hat{T}_n(s,\bm{z})\right|^2v(\bm{z})d\bm{z}ds,
\end{eqnarray*}
where $v:\R^d\to\R$ is some integrable weighting function. Applying Corollary \ref{cor:T_dach} and the continuous mapping theorem gives convergence in distribution of those test statistics. One can obtain distribution-free tests in the case of dimension $d=1$ as follows. Denote by $\{K_0(s,t):s\in[0,1], t\in \R\}$ a Kiefer-M\"uller process, i.e.\ a centered Gaussian process with covariance function $\Cov(K_0(s_1,t_1),K_0(s_2,t_2))=(s_1\mi s_2-s_1s_2)(t_1\mi t_2)$. Then $K_0(\cdot,\Sigma(\cdot))$ has the same distribution as $G_0(\cdot,\cdot)$.  Let further $\sigma (\cdot)$ be continuous and consider the consistent estimator 
$\hat{c}_n=n^{-1}\sum_{i=1}^{n}(Y_i-\hat{m}_n(X_i))^2\om_n(X_i)$
for $c=\int \si^2(u)f(u)du$. Applying a scaling property of the process $K_0$ in its second component and substitution in the integrals it is easy to derive convergence in distribution as follows, 
\begin{eqnarray*}
\frac{T_{n1}}{\hat c_n^{1/2}} &\underset{n\to\infty}{\overset{\cD}{\to}}&\sup\limits_{s\in[0,1],t\in[0,1]}\left|K_0(s,t)\right|,\quad
\frac{T_{n2}}{\hat c_n} \;\underset{n\to\infty}{\overset{\cD}{\to}}\; \sup\limits_{t\in[0,1]}\int_{0}^1\left|K_0(s,t)\right|^2ds,\\
\frac{T_{n3}}{\hat c_n^2} &\underset{n\to\infty}{\overset{\cD}{\to}}& \sup\limits_{s\in[0,1]}\int_0^1\left|K_0(s,t)\right|^2dt,\quad
\frac{T_{n4}}{\hat c_n^2} \;\underset{n\to\infty}{\overset{\cD}{\to}}\; \int_{0}^1\int_0^1\left|K_0(s,t)\right|^2dtds.
\end{eqnarray*}
For the latter two tests however the unknown weight function $v=\sigma^2 f$ needs to be chosen to obtain the limit as stated above. To obtain feasible asymptotically distribution-free tests, $T_{n3}$ and $T_{n4}$ should be replaced by 
\begin{eqnarray*}
\tilde{T}_{n3}&=& \sup\limits_{s\in[0,1]}\frac{1}{n}\sum\limits_{k=1}^{n}\left|\hat{T}_n(s,X_k)\right|^2\hat{\sigma}_n^2(X_k),\quad
\tilde{T}_{n4}\;=\;\int_0^1\frac{1}{n}\sum\limits_{k=1}^{n}\left|\hat{T}_n(s,X_k)\right|^2\hat{\sigma}_n^2(X_k)ds
\end{eqnarray*}
applying a nonparametric estimator for the variance function such as
$$\hat{\sigma}_n^2(x)=\dfrac{\sum_{j=1}^{n}K\Big(\frac{x-X_{j}}{h_n}\Big)(Y_{j}-\hat{m}_n(x))^2}{\sum_{j=1}^{n}K\Big(\frac{x-X_{j}}{h_n}\Big)}.$$
To conclude the section we will have a closer look at the alternative of one change point. For simplicity reasons we will only consider the test based on $T_{n1}$. To model the alternative we assume a triangular array
\begin{align*} 
Y_{n,t}=m_{n,t}(\bm{X}_{n,t})+U_{n,t}, \ t=1,\dots,n, 
\end{align*}
and validity of the alternative of one change point, i.e.\
\begin{align}\label{alternative}
H_{1}: \exists s_0\in (0,1): m_{n,t}(\cdot)=\begin{cases}m_{(1)}(\cdot), & t=1,\dots, \lf ns_0\rf \\ m_{(2)}(\cdot), & t=\lf ns_0\rf+1,\dots,n\end{cases} 
\end{align}
for some not further specified functions $m_{(1)}\not\equiv m_{(2)}$. Let $f_{n,t}$ denote the density of $\bm{X}_{n,t}$ and assume that for all $s\in(0,1]$ there exists a function $\bar{f}^{(s)}:\R^d\to\R$ such that
\begin{equation}\label{fs}
\lim\limits_{n\to\infty}\frac{1}{n}\sum\limits_{t=1}^{\lf ns \rf}f_{n,t}(\bm{x})=\bar{f}^{(s)}(\bm{x}), \ \forall \ \bm{x}\in\R^d.
\end{equation}
Under some regularity conditions it can be shown by applying \citeauthor{Kristensen20091433}'s (\citeyear{Kristensen20091433}) results that 
\begin{equation}\label{mn}
\sup\limits_{\bm{x}\in\bm{J}_n}\left|\hat{m}_n(\bm{x})-\bar{m}_n(\bm{x})\right|=o_P(1),
\end{equation}
where $\bar{m}_n(\bm{x})
=\sum_{i=1}^{n}f_{n,i}(\bm{x})m_{n,i}(\bm{x})/\sum_{i=1}^{n}f_{n,i}(\bm{x})$
converges to the mixture
\begin{equation}\label{mn-limit}
m_{(1)}(\bm{x})\frac{\bar{f}^{(s_0)}(\bm{x})}{\bar{f}^{(1)}(\bm{x})}+\left(1-\frac{\bar{f}^{(s_0)}(\bm{x})}{\bar{f}^{(1)}(\bm{x})}\right)m_{(2)}(\bm{x})
\end{equation}
of the regression functions before and after the change. Now for fixed $\bm{z}\in\R^d$ and $s\in(0,1)$ with $s\le s_0$, it holds that
\begin{align*}
\hat{T}_n(s,\bm{z})=\sqrt{n} \Delta(s,\bm{z})+o_P(\sqrt{n}),
\end{align*}
where 
\[\Delta(s,\bm{z})=\int\limits_{(-\bm{\infty},\bm{z}]}(m_{(1)}(\bm{u})-m_{(2)}(\bm{u}))\left(1-\frac{\bar{f}^{(s_0)}(\bm{u})}{\bar{f}^{(1)}(\bm{u})}\right)\bar{f}^{(s)}(\bm{u})d\bm{u}.\] 
As under $H_1$ this integral is non-zero for $s=s_0$ and some $\bm{z}$, convergence of $T_{n1}$ to infinity in probability and thus  consistency of the test can be deduced.

\begin{Remark} Consider the non-marked CUSUM process $\hat T_n(s,\bm{\infty})$ which is analogous to \citeauthor{Su2008347}'s (\citeyear{Su2008347}) procedure.  Considerations as above for the fixed alternative $H_1$ of one change point in $\lf ns_0\rf$ leads for $s\leq s_0$ to 
\begin{eqnarray*}
\Delta(s,\bm{\infty}) &=&\int(m_{(1)}(\bm{u})-m_{(2)}(\bm{u}))\left(1-\frac{\bar{f}^{(s_0)}(\bm{u})}{\bar{f}^{(1)}(\bm{u})}\right)\bar{f}^{(s)}(\bm{u})d\bm{u}\\
&=& s(1-s_0)\int (m_{(1)}(\bm{u})-m_{(2)}(\bm{u}))f(\bm{u})d\bm{u},
\end{eqnarray*}
where the last equality holds in case of a stationary covariate process. The integral  can be zero even if $m_{(1)}\neq m_{(2)}$. Then  tests based on the CUSUM process will not be consistent, while tests based on the marked CUSUM process are. We will consider some examples in section \ref{simus}. 
\end{Remark}

\section{A bootstrap procedure and the case of non-stationary variances}\label{bootstrap}

As alternative to the asymptotic test considered in section 3, in this section we will suggest a wild bootstrap approach. This resampling procedure can in particular be applied in the case of multivariate covariates, where the critical values for the asymptotic tests based on Corollary \ref{cor:T_dach} have to be estimated. Moreover, the bootstrap approach can be applied to obtain a test that detects changes in the conditional mean function, even when the conditional variance function is not stable. As desired, the test does not react sensitive to the unstable variance.
In contrast to the bootstrap approach, the limiting distribution from section 3 cannot be applied in the case of changes in the variance. 

We consider the model 
\begin{align*}
Y_{n,t}=m_{n,t}(\bm{X}_{n,t})+U_{n,t}, \ t=1,\dots,n, 
\end{align*}
with 
$E[U_{n,t}|\cF_n^{t}]=0$ and $E[U_{n,t}^2|\bm{X}_{n,t}]=\si^2_{n,t}(\bm{X}_{n,t})$  a.s.\
for some functions $\si^2_{n,t}:\R^d\to\R$ and $\cF_n^{t}:=\si(U_{n,j-1},\bm{X}_{n,j}:j\le t)$.  We assume $\bm{X}_{n,t}$ to be absolutely continuous with density function $f_{n,t}$.  The model considered in section \ref{sec-model} and the first part of section \ref{asympt} is the special case where $f_{n,t}(\cdot)=f(\cdot)$ and $\sigma^2_{n,t}(\cdot)=\sigma^2(\cdot)$ for all $t=1,\dots,n$ and for some $f,\si^2:\R^d\to\R$ not depending on $t$ and $n$. Both models allow for heteroscedasticity, but the more general model also allows for possible changes in $\si^2_{n,t}$, which should not effect the rejection probability of the test for 
\[H_0: m_{n,t}(\cdot)=m(\cdot), \ t=1,\dots,n,\]
(for some $m$ not depending on $t$ and $n$). We again consider the procedure
$$\hat{T}_n(s,\bm{z})=\frac{1}{\sqrt{n}}\sum_{i=1}^{\lf ns \rf}\hat{U}_{n,i}\om_n(\bm{X}_{n,i})I\{\bm{X}_{n,i}\le \bm{z}\}$$
with residuals $\hat{U}_{n,i}=Y_{n,i}-\hat{m}_n(\bm{X}_{n,i})$. Here $\hat{m}_n$ is defined as in \eqref{eq:NadWat}, but replacing $(\bm{X}_j,Y_j)$ by $(\bm{X}_{n,j}, Y_{n,j})$, $j=1,\dots,n$.
%
%

First define the wild bootstrap innovations as $U_{n,t}^*=\hat{U}_{n,t}\eta_t$,
where $\{\eta_t\}$ are i.i.d.\ random variables, independent of the original sample with $E[\eta_0]=0$, $E[\eta_0^2]=1$ and $E[\eta_0^4]<\infty$. Then the bootstrap data fulfilling the null hypothesis are generated by
\[Y^*_{n,t}=\hat{m}_n(\bm{X}_{n,t})+U_{n,t}^*.\]
Note that if the original data follow an autoregression model, say $d=1$ and $X_{n,t}=Y_{n,t-1}$, by the above choice the resulting bootstrap data does not follow the same structure. As was pointed out by \cite{Kreiss20123} this bootstrap data generation is still a reasonable choice in particular if the dependence structure of the underlying process does not show up in the asymptotic distribution. Another possibility might be a dependent wild bootstrap as suggested in \cite{Shao2010218}.

 The bootstrap residuals are defined as $\hat{U}^*_{n,t}=Y^*_{n,t}-\hat{m}^*_n(\bm{X}_{n,t})$,
where $\hat{m}^*_n$ is defined as $\hat m_n$ in \eqref{eq:NadWat}, but replacing $(\bm{X}_j,Y_j)$ by $(\bm{X}_{n,j}, Y_{n,j}^*)$, $j=1,\dots,n$. 
The bootstrap process is defined as
$$\hat{T}^*_n(s,\bm{z})=\frac{1}{\sqrt{n}}\sum_{i=1}^{\lf ns \rf}\hat{U}^*_{n,i}\om_n(\bm{X}_{n,i})I\{\bm{X}_{n,i}\le \bm{z}\}.$$
Bootstrap versions $T_{n\ell}^*$, $\ell=1,\dots,4$, are defined analogous to the test statistics $T_{n\ell}$, $\ell=1,\dots,4$, but based on $\hat T_n^*$ instead of $\hat T_n$. Then $H_0$ is rejected if $T_{n\ell}$ is larger than the ($1-\alpha$)-quantile of the conditional distribution of $T_{n\ell}^*$, given the original data. 

To motivate that we obtain a valid procedure (which holds the level asymptotically and is consistent) even in the case of changing variances, we will consider the limiting process $G_0$ of the original process $\hat{T}_n$ and the conditional limiting process $G_0^*$ of the bootstrap version $\hat{T}^*_n$ in subsections \ref{section-changing variances} and \ref{section-bootstrap} below. We will see that the processes $G_0$ and $G_0^*$ coincide under the null hypothesis. Note that some steps of the derivation are explained heuristically, whereas rigorously deriving the weak convergence would require a limit theorem for sequential empirical processes indexed in function classes for weakly dependent non-stationary data. Such a result is, to the best of our knowledge, not yet available in the literature and thus a rigorous proof is beyond the scope of the paper (see \cite{Mohr2017} for a related limit theorem that requires stationarity).

\subsection{Asymptotics for non-homogeneous variances}\label{section-changing variances}

Heuristically under $H_0$ one can proceed as in the proof of the first part of Theorem \ref{decomposition} in the beginning of section \ref{asympt}. Again one has the expansion $\hat T_n(s,\bm{z}) = A_{n2}(s,\bm{z})+A_{n1}(s,\bm{z})$, however, similar to Lemma \ref{Lemma2} in the appendix one will now obtain
\begin{align*}
A_{n1}(s,\bm{z}):=&\frac{1}{\sqrt{n}}\sum\limits_{i=1}^{\lf ns \rf} (m(\bm{X}_{n,i})-\hat{m}_n(\bm{X}_{n,i}))\om_n(\bm{X}_{n,i})I\{\bm{X}_{n,i}\le \bm{z}\}\\
=&\sqrt{n}\int_{(-\bm{\infty},\bm{z}]}(m(\bm{x})-\hat{m}_n(\bm{x}))\om_n(\bm{x})\bar{f}^{(s)}(\bm{x})d\bm{x}+o_P(1),
\end{align*}
 uniformly in $s\in[0,1]$ and $\bm{z}\in\R^d$ under suitable regularity conditions and under the assumption that the limit $\bar{f}^{(s)}$ as in (\ref{fs}) exists. Inserting the definition of $\hat m_n$ analogously to Lemmata \ref{Lemma:Teil mit m} and \ref{Lemma:Teil mit U} this will result in one negligible term and one term of the form
\begin{eqnarray*}
&&-\frac{1}{\sqrt{n}}\sum\limits_{j=1}^{n}U_{n,j}\int_{(-\bm{\infty},\bm{z}]}K_{h_n}(\bm{x}-\bm{X}_{n,j})\om_n(\bm{x})\frac{\bar{f}^{(s)}(\bm{x})}{\bar{f}^{(1)}(\bm{x})}d\bm{x}\\
&&= {}-\frac{1}{\sqrt{n}}\sum_{j=1}^{n}U_{n,j}\frac{\bar{f}^{(s)}(\bm{X}_{n,j})}{\bar{f}^{(1)}(\bm{X}_{n,j})}I\{\bm{X}_{n,j}\le \bm{z}\} +o_P(1).
\end{eqnarray*}
Further, Lemma \ref{Lemma1} will stay valid in analogous form. Thus, one obtains the exansion 
\[\hat{T}_n(s,\bm{z})=\Gamma_n(s,1,\bm{z})-\Gamma_n(1,s,\bm{z})+o_P(1)\]
with the process 
\begin{align*}
\Gamma_n(s,t,\bm{z}):=\frac{1}{\sqrt{n}}\sum\limits_{i=1}^{\lf ns \rf}U_{n,i}\bar{g}^{(t)}(\bm{X}_{n,i})I\{\bm{X}_{n,i}\le\bm{z}\}:s,t\in[0,1],\bm{z}\in\R^d,
\end{align*}
where $\bar{g}^{(t)}:=\bar{f}^{(t)}/\bar{f}^{(1)}$.
Now assume that the limit $\bar{h}^{(s)}(\cdot):=\lim_{n\to\infty}n^{-1}\sum_{i=1}^{\lf ns \rf}\si_{n,i}^2(\cdot)f_{n,i}(\cdot)$ exists for all $s\in(0,1]$ and that the process $\Gamma_n$ converges weakly to a centered Gaussian process $\Gamma$
(a proof would require a weak convergence result for sequential empirical processes indexed in general function classes and with a weakly dependent and non-stationary underlying triangular array process). 
The limiting covariance is
$$
E[\Gamma_n(s_1,t_1,\bm{z}_1)\Gamma_n(s_2,t_2,\bm{z}_2)]=\int_{(-\bm{\infty},\bm{z}_1\mi\bm{z}_2]}\bar{h}^{(s_1\mi s_2)}(\bm{u})\bar{g}^{(t_1)}(\bm{u})\bar{g}^{(t_2)}(\bm{u})d\bm{u}.
$$
Then with the continuous mapping theorem the weak convergence of $\hat{T}_n$ to a centered Gaussian process $\{G_0(s,\bm{z}):s\in[0,1],\bm{z}\in\R^d\}$ follows with covariances
\begin{align*}
&\Cov(G_0(s_1,\bm{z}_1),G_0(s_2,\bm{z}_2))\\
&=\hspace{-0.25cm}\int\limits_{(-\bm{\infty},\bm{z}_1\mi \bm{z}_2]}\hspace{-0.35cm}\left(\bar{h}^{(s_1\mi s_2)}(\bm{u})-\bar{h}^{(s_1)}(\bm{u})\bar{g}^{(s_2)}(\bm{u})-\bar{h}^{(s_2)}(\bm{u})\bar{g}^{(s_1)}(\bm{u})+\bar{h}^{(1)}(\bm{u})\bar{g}^{(s_1)}(\bm{u})\bar{g}^{(s_2)}(\bm{u})\right)\,d\bm{u}.
\end{align*}
Note that this is consistent with the stationary case as then $\bar{h}^{(s)}(\cdot)=s\si^2(\cdot)f(\cdot)$ and $\bar{g}^{(s)}(\cdot)=s$ and the same covariance function as in Corollary \ref{cor:T_dach} is obtained. 
 The convergence of the test statistics $T_{n\ell}$, $\ell=1,\dots,4$, in distribution follows again from the continuous mapping theorem. 

Under the change point alternative $H_1$ from (\ref{alternative}) with $m_{(1)}\not\equiv m_{(2)}$, analogous to the considerations in section \ref{asympt}  it holds that the test statistic $T_{n1}$ converges to infinity in probability.

\subsection{Derivations for the bootstrap process}\label{section-bootstrap}

Concerning the weak convergence of the bootstrap process $\hat T_n^*$, conditionally on the sample, we have again a look at the expansion in the beginning of section \ref{asympt} for the derivation of the first part of the proof of Theorem \ref{decomposition}. In what follows let $P^*$ denote the conditional probability and $E^*$ the conditional expectation, given the observations. Further let $Z_n=o_{P^*}(1)$ be short for $P^*(|Z_n|>\epsilon)=o_P(1)$ for all $\epsilon>0$.
Here we obtain 
$$\hat T_n(s,\bm{z}) = A_{n2}^*(s,\bm{z})+A_{n1}^*(s,\bm{z})$$
with 
$$A_{n2}^*(s,\bm{z}):=\frac{1}{\sqrt{n}}\sum\limits_{i=1}^{\lf ns \rf}U^*_{n,i}\om_n(\bm{X}_{n,i})I\{\bm{X}_{n,i}\le \bm{z}\}$$
and (similar to Lemma \ref{Lemma2} in the appendix)
\begin{align*}
A_{n1}^*(s,\bm{z}):=&\frac{1}{\sqrt{n}}\sum\limits_{i=1}^{\lf ns \rf} (\hat m_n(\bm{X}_{n,i})-\hat{m}_n^*(\bm{X}_{n,i}))\om_n(\bm{X}_{n,i})I\{\bm{X}_{n,i}\le \bm{z}\}\\
=&\sqrt{n}\int_{(-\bm{\infty},\bm{z}]}(\hat m_n(\bm{x})-\hat{m}_n^*(\bm{x}))\om_n(\bm{x})\bar{f}^{(s)}(\bm{x})d\bm{x}+o_{P^*}(1)
\end{align*}
with $\bar{f}^{(s)}$ as in (\ref{fs}). 
Inserting the definition of $\hat m_n^*$ this leads to a term (similar to Lemma \ref{Lemma:Teil mit m}) of the form
\[\frac{1}{\sqrt{n}}\sum\limits_{j=1}^{n}\int_{(-\bm{\infty},\bm{z}]}\left(\hat{m}_n(\bm{x})-\hat{m}_n(\bm{X}_{n,j})\right)K_{h_n}(\bm{x}-\bm{X}_{n,j})\om_n(\bm{x})\frac{\bar{f}^{(s)}(\bm{x})}{\bar{f}^{(1)}(\bm{x})}d\bm{x},\]
which is negligible, and a term (similar to Lemma \ref{Lemma:Teil mit U}) of the form 
\begin{align*}
&-\frac{1}{\sqrt{n}}\sum\limits_{j=1}^{n}U^*_{n,j}\int_{(-\bm{\infty},\bm{z}]}K_{h_n}(\bm{x}-\bm{X}_{n,j})\om_n(\bm{x})\frac{\bar{f}^{(s)}(\bm{x})}{\bar{f}^{(1)}(\bm{x})}d\bm{x}\\
&=-\frac{1}{\sqrt{n}}\sum\limits_{j=1}^{n}U^*_{n,j}\om_n(\bm{X}_{n,j})\frac{\bar{f}^{(s)}(\bm{X}_{n,j})}{\bar{f}^{(1)}(\bm{X}_{n,j})}I\{\bm{X}_{n,j}\le \bm{z}\}+o_{P^*}(1).
\end{align*}
Thus one obtains (under suitable regularity conditions) the expansion 
\begin{align*}
\hat{T}^*_n(s,\bm{z})
&=\Gamma^*_n(s,1,\bm{z})-\Gamma^*_n(1,s,\bm{z})+o_{P^*}(1),
\end{align*}
where 
\[\Gamma^*_n(s,t,\bm{z}):=\frac{1}{\sqrt{n}}\sum\limits_{i=1}^{\lf ns \rf}U^*_{n,i}\om_n(\bm{X}_{n,i})\bar{g}^{(t)}(\bm{X}_{n,i})I\{\bm{X}_{n,i}\le\bm{z}\}, \ s,t\in[0,1],\bm{z}\in\R^d,\]
and $\bar{g}^{(t)}$ is defined as in section \ref{section-changing variances}. In what follows we will assume that the process $\Gamma^*_n$, conditionally on the sample, converges weakly to a centered Gaussian process, in probability. Then, by the continuous mapping theorem,  $\hat{T}^*_n$, conditionally converges weakly to a centered Gaussian process, say $G_0^*$. We will calculate the asymptotic variances in order to show that under $H_0$ those coincide with the covariances of $G_0$ as in section \ref{section-changing variances}.
 First note that $E^*[U_{n,i}^*U_{n,j}^*]=\hat{U}_{n,i}^2I\{i=j\}$
almost surely. Under $H_0$ it holds that $\hat{U}_{n,t}=m(\bm{X}_{n,t})-\hat{m}_n(\bm{X}_{n,t})+U_{n,t}$ and $\hat{m}_n$ consistently estimates $m$, and thus
\begin{align*}
&E^*\left[\Gamma^*_n(s_1,t_1,\bm{z}_1)\Gamma^*_n(s_2,t_2,\bm{z}_2)\right]\notag\\
&=\frac{1}{n}\sum\limits_{i=1}^{\lf ns_1\rf \mi\lf ns_2\rf} \hat{U}_{n,i}^2\om_n(\bm{X}_{n,i})\bar{g}^{(t_1)}(\bm{X}_{n,i})\bar{g}^{(t_2)}(\bm{X}_{n,i})I\{\bm{X}_{n,i}\le \bm{z}_1\mi \bm{z}_2\}\notag\\
&=\frac{1}{n}\sum\limits_{i=1}^{\lf ns_1\rf \mi\lf ns_2\rf} U_{n,i}^2\om_n(\bm{X}_{n,i})\bar{g}^{(t_1)}(\bm{X}_{n,i})\bar{g}^{(t_2)}(\bm{X}_{n,i})I\{\bm{X}_{n,i}\le \bm{z}_1\mi \bm{z}_2\}+o_P(1)\\
&=E\left[\Gamma(s_1,t_1,\bm{z}_1)\Gamma(s_2,t_2,\bm{z}_2)\right]+o_P(1)
\end{align*}
under $H_0$, 
where $\Gamma$ is the limiting distribution of $\Gamma_n$ in section \ref{section-changing variances}. Thus, under $H_0$, $\hat{T}^*_n$ indeed (presumably) converges weakly to $G_0$   in probability, and thus the test statistic  $T_{n\ell}^*$ converges conditionally in distribution, to the same limits as  $T_{n\ell}$ (respectively for $\ell=1,\dots,4$). 

Under the alternative $H_1$ as  in (\ref{alternative}), $\hat{U}_{n,i}=m_{n,i}(\bm{X}_{n,i})-\hat{m}_n(\bm{X}_{n,i})+U_{n,i}$ and thus
it holds that
\begin{align*}
&E^*\left[\Gamma^*_n(s_1,t_1,\bm{z}_1)\Gamma^*_n(s_2,t_2,\bm{z}_2)\right]\notag\\
&=\frac{1}{n}\sum\limits_{i=1}^{\lf ns_1\rf \mi\lf ns_2\rf} \hat{U}_{n,i}^2\om_n(\bm{X}_{n,i})\bar{g}^{(t_1)}(\bm{X}_{n,i})\bar{g}^{(t_2)}(\bm{X}_{n,i})I\{\bm{X}_{n,i}\le \bm{z}_1\mi \bm{z}_2\}\notag\\
&=\frac{1}{n}\sum\limits_{i=1}^{\lf ns_1\rf \mi\lf ns_2\rf} U_{n,i}^2\om_n(\bm{X}_{n,i})\bar{g}^{(t_1)}(\bm{X}_{n,i})\bar{g}^{(t_2)}(\bm{X}_{n,i})I\{\bm{X}_{n,i}\le \bm{z}_1\mi \bm{z}_2\}\\
&{+r_{n1}}+r_{n2}
\end{align*}
for fixed $s_1,s_2,t_1,t_2\in[0,1]$ and $\bm{z}_1,\bm{z}_2\in\R^d$. The first term again converges in probability to $E\left[\Gamma(s_1,t_1,\bm{z}_1)\Gamma(s_2,t_2,\bm{z}_2)\right]$. It can further be shown that 
$$
r_{n1} =\frac{2}{n}\sum\limits_{i=1}^{\lf ns_1\rf \mi\lf ns_2\rf} \hspace{-0.4cm} U_{n,i}(m_{n,i}(\bm{X}_{n,i})-\hat{m}_n(\bm{X}_{n,i}))\om_n(\bm{X}_{n,i})\bar{g}^{(t_1)}(\bm{X}_{n,i})\bar{g}^{(t_2)}(\bm{X}_{n,i})I\{\bm{X}_{n,i}\le \bm{z}_1\mi \bm{z}_2\}$$
converges to zero in probability. 
However, 
\begin{align*}
r_{n2}=
&\frac{1}{n}\sum\limits_{i=1}^{\lf ns_1\rf \mi\lf ns_2\rf} (m_{n,i}(\bm{X}_{n,i})-\hat{m}_n(\bm{X}_{n,i}))^2\om_n(\bm{X}_{n,i})\bar{g}^{(t_1)}(\bm{X}_{n,i})\bar{g}^{(t_2)}(\bm{X}_{n,i})I\{\bm{X}_{n,i}\le \bm{z}_1\mi \bm{z}_2\}\\
&=\frac{1}{n}\sum\limits_{i=1}^{\lf ns_1\rf \mi\lf ns_2\rf}\hspace{-0.3cm} (m_{n,i}(\bm{X}_{n,i})-\bar{m}_n(\bm{X}_{n,i}))^2\om_n(\bm{X}_{n,i})\bar{g}^{(t_1)}(\bm{X}_{n,i})\bar{g}^{(t_2)}(\bm{X}_{n,i})I\{\bm{X}_{n,i}\le \bm{z}_1\mi \bm{z}_2\}\\
&+o_P(1),
\end{align*}
with the same $\bar{m}_n$ as in (\ref{mn}), which converges to 
\[(m_{(1)}(\bm{x})-m_{(2)}(\bm{x}))\bar{g}^{(s_0)}(\bm{x})+m_{(2)}(\bm{x})\] 
(see (\ref{mn-limit})). Thus, it can be shown that
\begin{align*}
r_{n2}
&=\frac{1}{n}\sum\limits_{i=1}^{\lf ns_1\rf \mi\lf ns_2\rf \mi \lf ns_0\rf} \left(m_{(1)}(\bm{X}_{n,i})-m_{(2)}(\bm{X}_{n,i})\right)^2\om_n(\bm{X}_{n,i})\left(1-\bar{g}^{(s_0)}(\bm{X}_{n,i})\right)^2\\
&\hspace{6.5cm}\cdot\bar{g}^{(t_1)}(\bm{X}_{n,i})\bar{g}^{(t_2)}(\bm{X}_{n,i})I\{\bm{X}_{n,i}\le \bm{z}_1\mi \bm{z}_2\}\\
&+\frac{1}{n}\sum\limits_{i=\lf ns_1\rf \mi\lf ns_2\rf \mi \lf ns_0\rf+1}^{\lf ns_1\rf \mi\lf ns_2\rf} \left(m_{(1)}(\bm{X}_{n,i})-m_{(2)}(\bm{X}_{n,i})\right)^2\om_n(\bm{X}_{n,i})\bar{g}^{(s_0)}(\bm{X}_{n,i})^2\\
&\hspace{6.5cm}\cdot\bar{g}^{(t_1)}(\bm{X}_{n,i})\bar{g}^{(t_2)}(\bm{X}_{n,i})I\{\bm{X}_{n,i}\le \bm{z}_1\mi \bm{z}_2\}\\
&+o_P(1).
\end{align*}

It can be seen that these terms do not vanish but converge to some limit in probability. Thus the limiting distribution $G_0^*$ under $H_1$ is not equal to $G_0$ and in particular depends on the changepoint $s_0$. As seen before  under $H_1$ the original test statistic $T_{n1}$ converges in probability to infinity. On the other hand, the bootstrap test statistic $T_{n1}^*$ conditionally converges in distribution to some non-degenerated limit, in probability. Thus the bootstrap test is consistent. 

\section{Finite sample properties}  \label{simus}
A small Monte Carlo study is conducted in order to compare the results for $T_{n1}$ and $T_{n2}$ from section \ref{asympt} with those of the traditional CUSUM versions denoted by $KS:=\sup_{s\in[0,1]}|\hat{T}_n(s,\infty)|$ and $CM:=\int|\hat{T}_n(s,\infty)|^2ds$. Note that the results for $\tilde{T}_{n3}$ and $\tilde{T}_{n4}$ are similar and omitted for reasons of brevity. Asymptotic tests are applied to data satisfying models 1 and 2, while the bootstrap versions are applied to model 3 explained below. Also note that simulation results for a multidimensional autoregression model can additionally be found in the supplement. All simulations are carried out with a level of $5\%$, $500$ replications and $200$ bootstrap replications and for sample sizes $n\in\{100,300,500\}$. For the nonparametric estimators we use a fourth order Epanechnikov kernel and the bandwidth is chosen by the cross validation method. For simplicity we set $\om_n\equiv 1$. The data is simulated from the following models.
\begin{alignat*}{3}
&&\text{(model 1)} 
\hspace{2cm}&Y_t=m_t(X_t)+\sqrt{1+0.5X_t^2}\ep_t, \ \ep_t\sim\mathcal{N}(0,1), & \\
&~&&m_t(x)=\begin{cases}0.5x, \ & t=1,\dots, \lf n/2\rf \\(0.5+\Delta_0e^{-0.8x^2})x, \ & t=\lf n/2 \rf +1,\dots, n\end{cases},
\end{alignat*}
where $X_i$ is an exogenous variable following the AR(1) model $X_t=0.4X_{t-1}+\xi_t$ with $\xi_i$ being i.i.d.~$\sim\mathcal{N}(0,1)$ and $\Delta_0\in \{0,0.5,1,1.5,1.5,2,2.5,3,3.5,4\}$. 
\begin{alignat*}{3}
&&\text{(model 2)} 
\hspace{1.8cm}&Y_t=m_t(Y_{t-1})+\si(Y_{t-1})\ep_t,  \ \ep_t\sim\mathcal{N}(0,1),& \\
&~&&m_t(x)=\begin{cases}-0.9x, \ & t=1,\dots, \lf n/2\rf \\\left(-0.9+\Delta_0\right)x, \ & t=\lf n/2 \rf +1,\dots, n\end{cases}, \hspace{1.2cm}
\end{alignat*}
with $\Delta_0\in \{0,0.2,0.4,0.6,0.8,1,1.2,1.4,1.6,1.8\}$. Consider the homoscedastic case, where $\si^2(x)=1$ and the heteroscedastic case, where $\si^2(x)=1+0.1x^2$.
\begin{alignat*}{3}
&&\text{(model 3)} 
\hspace{1.8cm}&Y_t=m_t(Y_{t-1})+\si_t(Y_{t-1})\ep_t, \ \ep_t\sim\mathcal{N}(0,1), & \\
&~&&\sigma^2_t(x)=\begin{cases}1+0.1x^2, \ & t=1,\dots,\lf nt_0 \rf \\ 1+0.8x^2, \ & t=\lf nt_0 \rf +1,\dots,n\end{cases},\\
&~&&m_t(x)=\begin{cases}0.9x, \ & t=1,\dots, \lf n/2\rf \\\left(0.9-\Delta_0\right)x, \ & t=\lf n/2\rf +1,\dots, n\end{cases}, \hspace{1.2cm}
\end{alignat*}
with $\Delta_0\in\{0,1.3\}$ and $t_0\in\{0.25,0.5,0.75\}$.

Model 1 is a regression model with autoregressive covariables. In model 2 we consider both a homoscedastic and heteroscedastic autoregression model, while model 3 is a heteroscedastic autoregression with non-homogeneous variances. All models fulfill $H_0$ for $\Delta_0=0$ and $H_1$ for $\Delta_0\neq 0$ with a change in regression function occurring in $\lf n/2\rf$. Further, note that models 1 and 2 fulfill the stationarity and mixing assumptions under $H_0$, while model 3 is not stationary as one change occurs in the conditional variance function under both $H_0$ and $H_1$. Hence, for model 3 we apply the bootstrap test from section \ref{bootstrap}.

Figure \ref{fig:Vis_mean_exo}, \ref{fig:Vis_AR1} and \ref{fig:Vis_AR1ARCH1_small} are visualizations of the performance of $T_{n1}$ and $T_{n2}$, as well as $KS$ and $CM$ in model 1 and 2. Under the null the rejection frequencies for all tests are near the nominal level. For model 1 the CUSUM tests are not consistent against $H_1$, while the tests based on the marked process are. In model 2 the rejection frequencies of all tests increase with increasing break size. Note however that the increase is much faster for $T_{n1}$ and $T_{n2}$ than for the CUSUM tests. Also note that the influence of the conditional variance is rather small resulting in a similar performance in both the homoscedastic and heteroscedastic case. Table \ref{sim_nonstat_var_all} shows the rejection frequencies of the bootstrap procedure using $T^*_{n1}$ and $T^*_{n2}$, as well as the bootstrap version of the CUSUM tests $KS$ and $CM$ under both the null and the alternative hypothesis. The level simulations show that all tests perform reasonably well under $H_0$, approximately holding the level indicating that the bootstrap test is -- as desired -- not sensitive to changes in the conditional variance function. Furthermore, it can be seen that for all models and all tests the rejection frequency under $H_1$ exceeds the level, indicating that the change point is detected. With increasing sample size, the number of rejections increases rapidly for $T_{n1}^*$ and $T_{n2}^*$, while it stays approximately constant for the bootstrap versions of $KS$ and $CM$. This is presumably due to the fact that the test statistics based on $\hat{T}_{n}(s,\infty)$ estimate some integral that might be small under $H_1$. As was pointed out in subsection \ref{section-bootstrap}, this integral not vanishing is essential for the consistency property for the bootstrap tests.
%
\begin{figure}[!htbp]
  \centering
     \includegraphics[width=1\textwidth]{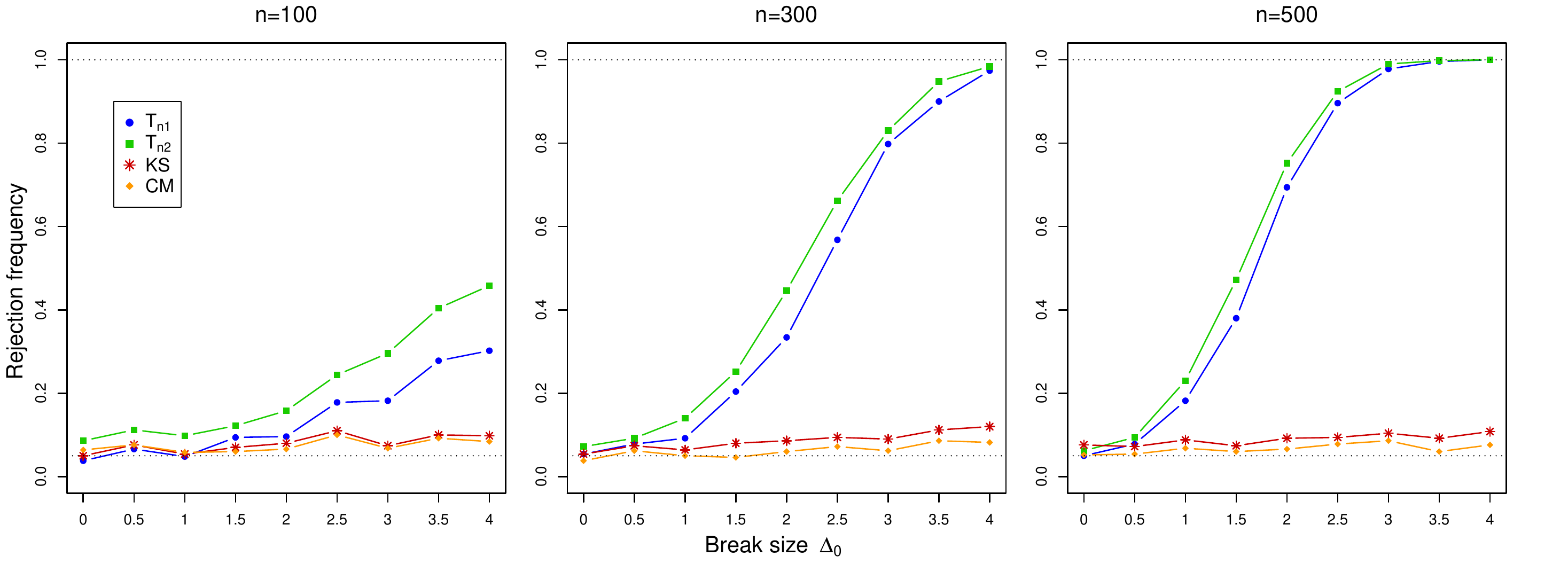}
  \caption{Rejection frequencies in model 1}
  \label{fig:Vis_mean_exo}
\end{figure}
\begin{figure}[!htbp]
  \centering
     \includegraphics[width=1\textwidth]{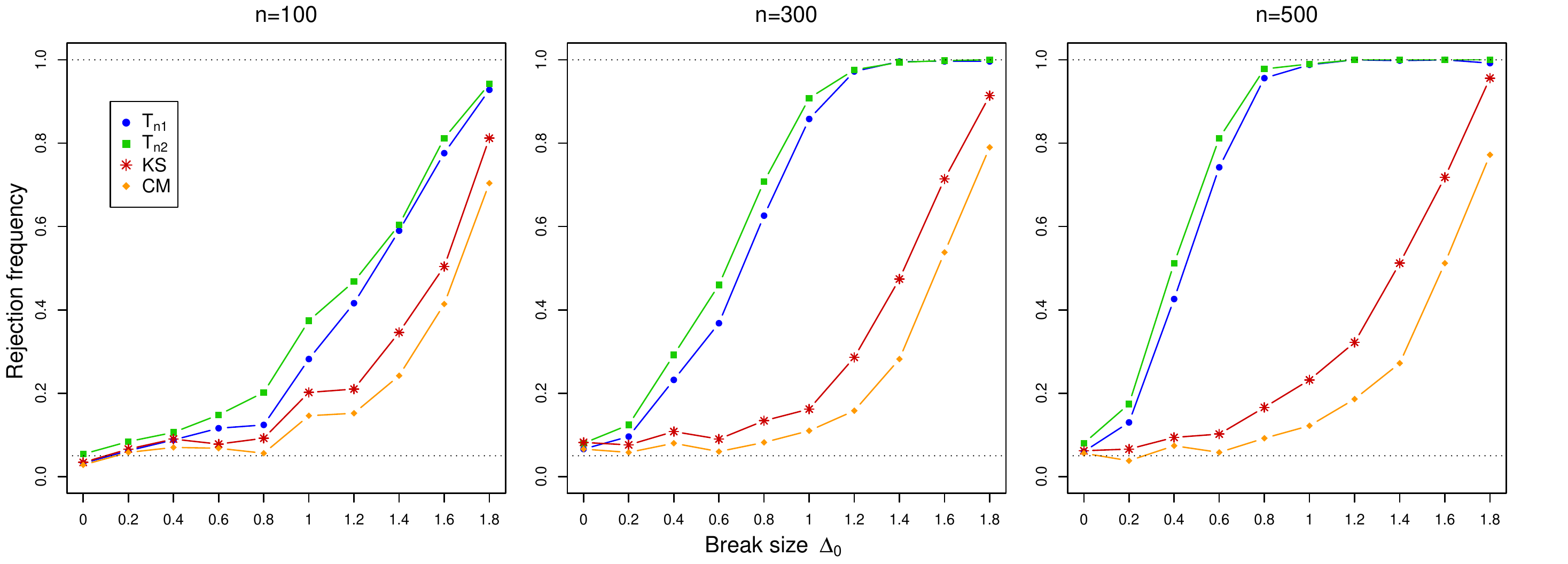}
  \caption{Rejection frequencies in model 2 with $\si^2(x)=1$}
  \label{fig:Vis_AR1}
\end{figure}
\begin{figure}[!htbp]
  \centering
     \includegraphics[width=1\textwidth]{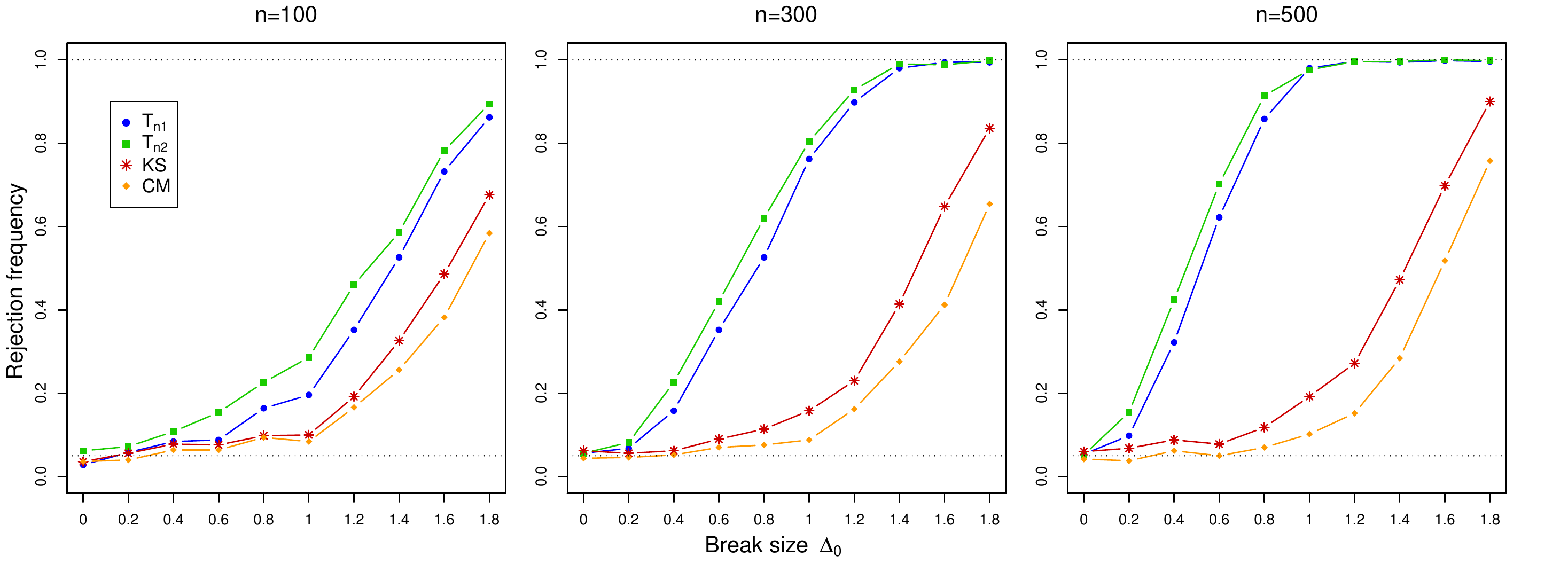}
  \caption{Rejection frequencies in model 2 with $\si^2(x)=1+0.1x^2$}
  \label{fig:Vis_AR1ARCH1_small}
\end{figure}
\begin{table}[!htbp]
\small
  \centering 
	\caption{Rejection frequencies in model 3}\label{sim_nonstat_var_all}
\begin{tabular}{@{}r r c r r r r c r r r r @{}} 
    \toprule
\multicolumn{2}{c}{}& \phantom{a} & \multicolumn{4}{c}{under $H_0$} & \phantom{a} & \multicolumn{4}{c}{under $H_1$}\\
\cmidrule{4-7} \cmidrule{9-12}
 $t_0$& $n$  &  & $T_{n1}^*$   & $T_{n2}^*$  &  $KS$ & $CM$  &  & $T_{n1}^*$   & $T_{n2}^*$  &  $KS$ & $CM$\\ 
\midrule 
$0.25$   & $100$  & &  		 $0.030$  &  $0.046$  &	 $0.030$  &  $0.054$  	& &		$0.286$  &  $0.270$  &  $0.192$  &  $0.168$	 \\
			   & $300$  & &      $0.068$  &  $0.064$  &  $0.080$  &  $0.052$    & &   $0.652$  &  $0.644$  &  $0.248$  &  $0.172$  \\		
				 & $500$  & &      $0.060$  &  $0.052$  &  $0.058$  &  $0.046$    & &   $0.878$  &  $0.868$  &  $0.264$  &  $0.194$  \\
\midrule 
$0.50$   & $100$  & &  		$0.068$  &  $0.048$  &  $0.068$  &  $0.056$			&	&	 	$0.420$  &  $0.438$  &  $0.316$  &  $0.256$  \\
			   & $300$  & &  		$0.066$  &  $0.050$  &  $0.056$  &  $0.046$   	&	&   $0.868$  &  $0.894$  &  $0.378$  &  $0.292$  \\		
				 & $500$  & &  		$0.046$  &  $0.040$  &  $0.058$  &  $0.040$			& &	  $0.994$  &  $0.996$  &  $0.434$  &  $0.324$ 	\\		
\midrule 
$0.75$   & $100$  & &  		$0.060$  &  $0.056$  &  $0.072$  &  $0.070$			&	&  	$0.404$  &  $0.388$  &  $0.332$  &  $0.266$  \\
			   & $300$  & &  		$0.048$  &  $0.048$  &  $0.050$  &  $0.056$		  &	&   $0.830$  &  $0.848$  &  $0.382$  &  $0.250$  \\		
				 & $500$  & &  		$0.034$  &  $0.040$  &  $0.046$  &  $0.056$			&	&   $0.986$  &  $0.988$  &  $0.350$  &  $0.202$  \\				
\bottomrule  
\end{tabular}%
\end{table}

Finally, we apply the asymptotic test based on $T_{n1}$ to $36$ measurements of the annual flow volume of the small Czech river R\'{a}ztoka recorded between 1954 and 1989. It was considered by \cite{Huskova2003201}. We set $X_t$ as the annual rainfall and $Y_t$ as the annual flow volume. The asymptotic test clearly rejects $H_0$ with a $p$-value of $0.0006$. The possible change point is estimated by $\hat{s}_n$ from section \ref{conclusion} and suggests a change in 1979. Note that this is consistent with the literature. As was pointed out by \cite{Huskova2003201} deforestation had started around that time, which is a possible explanation. Figure \ref{fig:creek_plot} shows on the left-hand side the scatterplot $X_t$ against $Y_t$ using dots for the observations after the estimated change and crosses for the observations before the estimated change. On the right-hand side the figure shows the cumulative sum, $\sup_{z\in\R}|\hat{T}_n(\cdot,z)|$, as well as the critical value (red horizontal line) and the estimated change (green vertical line).

\begin{figure}[!htbp]
  \centering
     \includegraphics[width=0.9\textwidth]{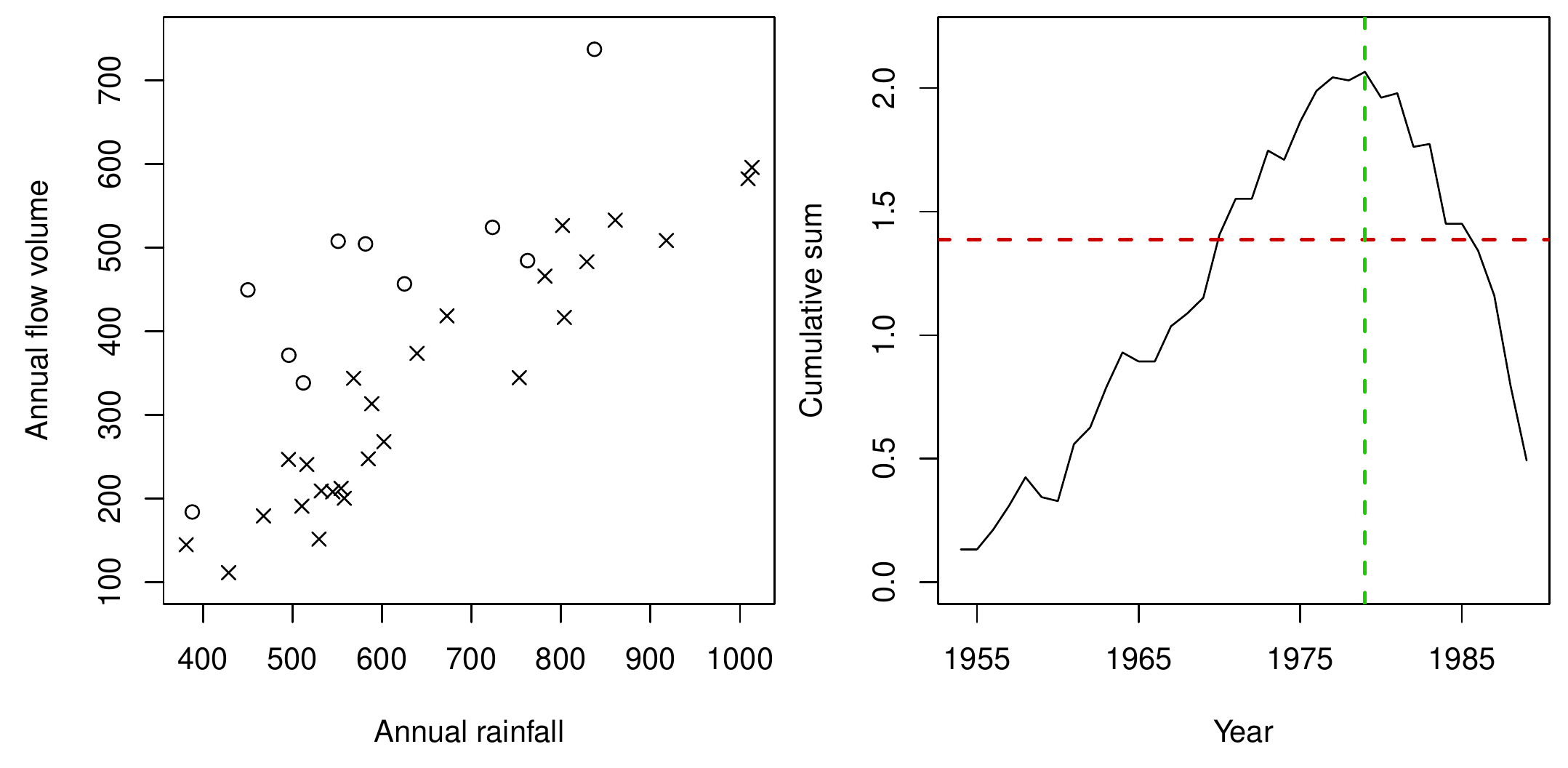}
  \caption{R\'{a}ztoka data: scatterplot (left) and CUSUM (right)}
  \label{fig:creek_plot}
\end{figure}

\section{Concluding remarks} \label{conclusion}

We suggested a new test for structural breaks in the regression function in nonparametric time series (auto-)regression. Our approach combines CUSUM statistics with the marked empirical process approach from goodness-of-fit testing. The considered model is very general and does not require independent innovations, nor homoscedasticity. We show favorable asymptotic properties and demonstrate that the new testing procedures are consistent against fixed alternatives, while the traditional CUSUM tests are not. 
An estimator for the change point is given by $\hat{s}_n:=\arg\max_{s\in[0,1]}\sup_{\bm{z}\in\R^d}|\hat{T}_n(s,\bm{z})|$. Asymptotic properties of this estimator will be considered in future research. 

Moreover we have suggested a bootstrap version that can also be applied to detect changes in the regression function in the presence of changing variance functions. In a forthcoming paper we will consider testing for changes in the variance function. 

\appendix

\section{Proofs and derivations} \label{appendix}

In subsection \ref{Appendix-nonparametric} we give some auxiliary results for the proof of Theorem \ref{decomposition}. The proof of the first part of the theorem was given in the main text, while the proof of the second part can be found in  subsection \ref{Appendix-proof}. Some lemmata are proved in subsection \ref{Lemma-proof} while some details are deferred to the supplement. 
Detailed proofs can also be found in \cite{Mohr2018}.  

\subsection{Auxiliary results}\label{Appendix-nonparametric}


The following assumptions are formulated for the first lemma that gives uniform rates of convergence for the regression estimator $\hat m_n$ from \eqref{eq:NadWat} and its derivatives. They hold under the assumptions of Theorem \ref{decomposition}. 

\begin{enumerate}
\item[\textbf{(P)}] Let $(Y_t,\bm{X}_t)_{t\in\Z}$ be a strictly stationary and strongly mixing process with mixing coefficient $\alpha(\cdot)$. For some $b>2$ let
$\alpha(t)=O(t^{-\beta})$ for $t\to\infty$ with some $\beta>(1+(b-1)\left(1+d\right))/(b-2)$. 
\item[\textbf{(B3)}] With $b$ and $\beta$ from assumption \textbf{(P)} let 
$(\log n)/(n^{\theta}h_n^d)=o(1)$ for $\theta=(\beta-1-d-(1+\beta)/(b-1))/(\beta+3-d-(1+\beta)/(b-1))$. 
\end{enumerate}

\begin{Lemma} \label{Raten kernel}
Under the assumptions \textbf{(P)}, \textbf{(M)}, \textbf{(J)}, \textbf{(F1)}, \textbf{(K)}, \textbf{(B1)} and \textbf{(B3)} the following rates of convergence can be obtained for the Nadaraya-Watson estimator $\hat{m}_n$, 
\begin{enumerate}
\item[(a)] $\sup\limits_{\bm{x}\in \bm{J}_n}\left|\hat{m}_n(\bm{x})-m(\bm{x})\right|=O_P\left(\left(\sqrt{\frac{\log{(n)}}{nh_n^d}}+h_n^rp_n\right)q_n\delta_n\right)$,
\item[(b)] $\sup\limits_{\bm{x}\in \bm{J}_n}\left|D^{\bm{k}}\left(\hat{m}_n(\bm{x})-m(\bm{x})\right)\right|=O_P\left(\left(\sqrt{\frac{\log{(n)}}{nh_n^{d+2|\bm{k}|}}}+h_n^rp_n\right)p_n^{|\bm{k}|}q_n\delta_n^{|\bm{k}|+1}\right)$ for all $\bm{k}\in\N_0^d$ with $1\le |\bm{k}|\le l+1$,
\item[(c)] $\displaystyle\sup\limits_{\substack{\bm{x},\bm{y}\in \bm{J}_n \\ \bm{x}\neq\bm{y}}}\frac{\left|D^{\bm{k}}\left(\hat{m}_n(\bm{x})-m(\bm{x})\right)-D^{\bm{k}}\left(\hat{m}_n(\bm{y})-m(\bm{y})\right)\right|}{\|\bm{x}-\bm{y}\|^{\eta}}=o_P(1)$ for all $\bm{k}\in\N_0^d$ with $|\bm{k}|= l$. 
\end{enumerate}
\end{Lemma}

The proof of Lemma \ref{Raten kernel} is analogous to the proof of Theorem 8 of \cite{Hansen2008726} and omitted for the sake of brevity. The proofs of the following lemmata are given in subsection \ref{Lemma-proof}. 

\begin{Lemma} \label{Lemma2}
Under the assumptions of Theorem \ref{decomposition} (i) and under $H_0$ we have for $A_{n1}$ from (\ref{A_n1})
\begin{align*}
A_{n1}(s,\bm{z})
&=s\sqrt{n}\int_{\R^d}(m(\bm{x})-\hat{m}_n(\bm{x}))\om_n(\bm{x})I\{\bm{x}\le \bm{z}\}f(\bm{x})d\bm{x}+o_P(1)
\end{align*}
 uniformly in $s\in [0,1]$ and $\bm{z}\in\R^d$.
\end{Lemma}

\begin{Lemma} \label{Lemma1}
Under the  assumptions of Theorem \ref{decomposition} (i) and under $H_0$  we have for $A_{n2}$ from (\ref{A_n2}) and  $T_n$ from (\ref{Tn})
\begin{align*}
A_{n2}(s,\bm{z})=T_n(s,\bm{z})+o_P(1),
\end{align*}
uniformly in $s\in [0,1]$ and $\bm{z}\in\R^d$.
\end{Lemma}

\begin{Lemma} \label{Lemma:Teil mit m}
Under the assumptions of Theorem \ref{decomposition} (i) and under $H_0$
\begin{align*}
\frac{1}{\sqrt{n}}\sum\limits_{i=1}^{n}\int_{(-\bm{\infty},\bm{z}]}\left(m(\bm{y})-m(\bm{X}_i)\right)K_{h_n}(\bm{y}-\bm{X}_i)\om_n(\bm{y})\frac{f(\bm{y})}{\hat{f}_n(\bm{y})}d\bm{y}=o_P(1)
\end{align*}
holds uniformly in $\bm{z}\in\R^d$.
\end{Lemma}

\begin{Lemma} \label{Lemma:Teil mit U}
Under the  assumptions of Theorem \ref{decomposition} (i) and under $H_0$
\begin{align*}
\frac{1}{\sqrt{n}}\sum\limits_{i=1}^{n}U_i\left(\int_{(-\bm{\infty},\bm{z}]}K_{h_n}(\bm{y}-\bm{X}_i)\om_n(\bm{y})\frac{f(\bm{y})}{\hat{f}_n(\bm{y})}d\bm{y}-\om_n(\bm{X}_i)I\{\bm{X}_i\le\bm{z}\}\right)=o_P(1)
\end{align*}
holds uniformly in $\bm{z}\in\R^d$.
\end{Lemma}

\subsection{Proof of Theorem \ref{decomposition}(ii)}\label{Appendix-proof}

 For the proof of the second part of Theorem \ref{decomposition} we use a recent result on weak convergence of sequential empirical processes indexed in function classes that can be applied for strongly mixing sequences, see \cite{Mohr2017}. It is stated in Lemma \ref{Mohr2017} and uses the following notion of bracketing number.

\begin{Definition}[Bracketing number]\label{brackets}
Let $\cX$ be a measure space, $\cF$ some class of functions $\cX\to\R$ and $\rho$ some semi norm on $\cF$. For all $\ep>0$, let $N=N(\ep)$, be the smallest integer, for which there exist a class of functions $\cX\to\R$, denoted by $\cB$ and called \textit{bounding class} and a function class $\cA \subset \cF$ called \textit{approximating class} such that $|\cB|=|\cA|=N$, $\rho(b)<\ep,\ \forall \ b\in\cB$ and for all $\ph\in\cF$ there exist  $a^*\in\cA$ and  $b^*\in\cB$ such that $|\ph-a^*|\le b^*$. Then $N(\ep)$ is called the \textit{bracketing number} and denoted by $\tilde{N}_{[~]}(\ep,\cF,\rho)$. 
\end{Definition}

Note that the usual notion for bracketing number (as in Definition 2.1.6 in \cite{vanderVaart1996}) will be referred to as $N_{[~]}(\ep,\cF,\rho)$. 

\begin{Lemma}[Corollary 2.7 in \cite{Mohr2017}]\label{Mohr2017}
Let $\{X_t:t\in\Z\}$ be a strictly stationary sequence of random variables with values in some measure space $\cX$. Let $\cF$ be a class of measurable functions $\cX\to\R$. Let furthermore the following assumptions hold.
\begin{enumerate}
\item[\textbf{(A1)}] Let $\{X_t:t\in\Z\}$ be strongly mixing, such that $\sum_{t=1}^{\infty}t^{Q-2}\alpha(t)^{\gamma/(2+\gamma)}<\infty$ for some $\gamma>0$ and even $Q> 2$.
\item[\textbf{(A2)}] For $Q$ and $\gamma$ from assumption \textbf{(A1)} let $\int_{0}^{1}x^{-\gamma/(2+\gamma)}(\tilde{N}_{[~]}(x,\cF,\|\cdot\|_{L_2(P)}))^{1/Q}dx<\infty$, where $X_1\sim P$. Furthermore, assume that each $\ep>0$ allows for a choice of bounding class $\cB$, such that $E\left[|b(X_{1})|^{i(2+\gamma)/2}\right]^{1/2}\le\ep$ for all $b\in\cB$ and for all $i=2,\dots,Q$.
\item[\textbf{(A3)}] Let $\cF$ possess an envelope function $F$, with $E[|F(X_1)|^Q]<\infty$ and let there exist a constant $L<\infty$, such that $\sup_{\ph\in\cF}E\left[ |\ph(X_1)|^{Q(2+\gamma)/2}\right]\le L$.
\end{enumerate} 
Furthermore, let for all $K\in\N$ and all finite collections $\ph_k\in\cF$, $s_k\in[0,1]$, $k=1,\dots,K$, $\left(G_n(s_k,\ph_k)\right)_{k=1,\dots,K}\overset{\cD}{\rightarrow} \left(G(s_k,\ph_k)\right)_{k=1,\dots,K}$, where $G_n(s,\ph)=n^{-1/2}\sum_{i=1}^{\lf ns \rf} (\ph(X_i)-E[\ph(X_i)])$ for $s\in[0,1],\ph\in\cF$ and $G=\{G(s,\ph):s\in[0,1],\ph\in\cF\}$ is a centered Gaussian process. \\

\noindent Then $\left\{G_n(s,\ph):s\in[0,1],\ph\in\cF\right\}$ converges weakly to $G$ in $\ell^{\infty}([0,1]\times \cF)$.
\end{Lemma}

\begin{proof}[Proof of Theorem \ref{decomposition}(ii)]
First notice that due to assumption \textbf{(G)} and under the null hypothesis $(U_t,\bm{X}_t)_{t\in\Z}$ is a strictly stationary sequence of random variables with values in $\R\times\R^d$. Denote by $P$ the common marginal distribution of $(U_1,\bm{X}_1)$ and define $\cF:=\{(u,\bm{x})\mapsto uI\{\bm{x}\le\bm{z}\}:\bm{z}\in\R^d\}$. The convergence of $T_n$ is then implied by the weak convergence of
\[G_n:=\Big\{G_n(s,\ph):=\frac{1}{\sqrt{n}}\sum\limits_{i=1}^{\lf ns\rf}\Big(\ph(U_i,\bm{X}_i)-\int \ph dP\Big):s\in[0,1],\ph\in\cF\Big\}\] 
in $\ell^{\infty}([0,1]\times\cF)$. We apply Lemma \ref{Mohr2017}. 
Condition \textbf{(A1)} on the mixing coefficient of $(U_t,\bm{X}_t)_{t\in\Z}$ is implied by assumption \textbf{(G)} on the mixing coefficient of $(Y_t,\bm{X}_t)_{t\in\Z}$ and the null hypothesis as measurable functions maintain mixing properties. To show condition \textbf{(A2)} on the function class $\cF$, the choice of approximating functions and bounding functions, as in Definition \ref{brackets}, will be discussed in more detail. 
Denote with $\bar{c}$ from assumption \textbf{(U)}, $h(\bm{x})=\bar{c}(\bm{x})f(\bm{x})$ and $H(\bm{x})=\int_{(-\bm{\infty},\bm{x}]}h(\bm{t})d\bm{t}$ for $\bm{x}\in\R^d$ and for all $i=1,\dots, d$ and $x\in\R$,
%
\begin{align*}
h_i(x)=\int\dots\int h(x_1,\dots,x_{i-1},x,x_{i+1},\dots,x_d)dx_1\dots dx_{i-1}dx_{i+1}\dots dx_d
\end{align*}
and $H_i(x)=\int_{-\infty}^{x}h_i(t)dt$.
Let $\ep>0$ and choose for all $i=1,\dots,d$ some $N_i=N_i(\ep)\in\N$ and $-\infty=z_{0,i}<\dots<z_{N_i,i}=\infty$, such that
\begin{align}
H_i\left(z_{j_i,i}\right)-H_i\left(z_{j_i-1,i}\right)\le\frac{\ep^2}{d}, \ \forall \ j_i=1,\dots,N_i, \ i=1,\dots,d. \label{eq:brackets for 1}
\end{align}
Since $H_i$ is continuous and $H_i(-\infty)=H(-\bm{\infty})=0$ and $H_i(\infty)=H(\bm{\infty})\le M$ for $M<\infty$ from assumption \textbf{(U)}, $N_i$ can be chosen to be smaller than $2dM\ep^{-2}$ for all $i=1,\dots,d$. By using cartesian products, a partition of $\R^d$ is obtained. For simplicity reasons the following notation will be used. For $\bm{j}=(j_1,\dots,j_d)\in\N^d$ let 
$\bm{z}_{\bm{j}}:=(z_{j_1,1},\dots,z_{j_d,d})$, 
and $\bm{j-1}:=(j_1-1,\dots,j_d-1)\in\N^d$. For all $\bm{j}\in\cart_{i=1}^{d}\{1,\dots,N_i\}$ define approximating functions 
\[a_{\bm{j}}(u,\bm{x}):= uI\left\{\bm{x}\le \bm{z}_{\bm{j}}\right\}\]
and bounding functions
\[b_{\bm{j}}(u,\bm{x}):= |u|\left(I\left\{\bm{x}\le \bm{z}_{\bm{j}}\right\}-I\left\{\bm{x}\le \bm{z}_{\bm{j}-\bm{1}}\right\}\right).\]
Notice that $a_{\bm{j}}\in\cF$ while $b_{\bm{j}}\notin\cF$ for all $\bm{j}\in\cart_{i=1}^{d}\{1,\dots,N_i\}$. For each $\bm{z}\in\R^d$ there exists a $\bm{j}\in\cart_{i=1}^{d}\{1,\dots,N_i\}$ such that $\bm{z}\in(\bm{z}_{\bm{j-1}},\bm{z}_{\bm{j}}]$. Therefore for each $\ph\in\cF$ there exists a $\bm{j}\in\cart_{i=1}^{d}\{1,\dots,N_i\}$ such that 
$|\ph-a_{\bm{j}}|\le b_{\bm{j}}$.
Further for $\bm{j}\in\cart_{i=1}^{d}\{1,\dots,N_i\}$ it holds that 
\begin{align*}
\left\|b_{\bm{j}}\right\|_{L_2(P)}^2
&=E\left[|U_t|^2\left(I\left\{\bm{X}_t\le \bm{z}_{\bm{j}}\right\}-I\left\{\bm{X}_t\le \bm{z}_{\bm{j}-\bm{1}}\right\}\right)\right]
=\int_{(-\bm{\infty},\bm{z}_{\bm{j}}]\backslash (-\bm{\infty},\bm{z}_{\bm{j}-\bm{1}}]}\si^2(\bm{u})f(\bm{u})d\bm{u}\\
&\leq H(\bm{z}_{\bm{j}})-H(\bm{z}_{\bm{j-1}}) \le \sum\limits_{i=1}^{d}\left(H_i\left(z_{j_i,i}\right)-H_i\left(z_{j_i-1,i}\right)\right)\le\ep^2
\end{align*}
due to \eqref{eq:brackets for 1}. Furthermore for all $i=2,\dots,Q$ by Jensen's inequality and \textbf{(U)}, it holds that
\[E\left[|U_t|^{i\frac{2+\gamma}{2}}|\bm{X}_t\right]\le E\left[|U_t|^{i\frac{2+\gamma}{2}}|\bm{X}_t\right]^{\frac{i}{Q}}\le (c(\bm{X}_t)^Q)^{\frac{i}{Q}}=c(\bm{X}_t)^i \text{ a.s.},\]
and thus 
\begin{align*}
\int|b_{\bm{j}}|^{i\frac{2+\gamma}{2}}dP
&=E\left[|U_t|^{i\frac{2+\gamma}{2}}\left(I\left\{\bm{X}_t\le \bm{z}_{\bm{j}}\right\}-I\left\{\bm{X}_t\le \bm{z}_{\bm{j}-\bm{1}}\right\}\right)\right]\\
&\leq H(\bm{z}_{\bm{j}})-H(\bm{z}_{\bm{j-1}})\leq\ep^2. 
\end{align*}
Since $N_i=O(\ep^{-2})$ for all $i=1,\dots,d$, it holds that 
$\tilde{N}_{[~]}(\ep,\cF,\|\cdot\|_{L_2(P)})=O\left(\ep^{-2d}\right)$.
As $Q>d(2+\gamma)$ holds, assumption \textbf{(A2)} from Lemma \ref{Mohr2017} is therefore satisfied. Assumption \textbf{(A3)} is also satisfied as $\bar{F}:\R\times\R^d\to \R, \ (u,\bm{x})\mapsto u$ is an envelope function of $\cF$ such that 
\begin{align*}
\int |\bar{F}|^{Q}dP
&=E\left[|U_t|^Q\right]
\le E\left[|U_t|^{Q\frac{2+\gamma}{2}}\right]^{\frac{2}{2+\gamma}}
\le \left(\int c(\bm{u})^Qf(\bm{u})d\bm{u}\right)^{\frac{2}{2+\gamma}}
<\infty,
\end{align*}
and additionally, it holds that
\begin{align*}
\sup\limits_{\ph\in\cF}\int |\ph|^{Q\frac{2+\gamma}{2}}dP
&=\sup\limits_{\bm{z}\in\R^d}E\left[|U_t|^{Q\frac{2+\gamma}{2}}I\{\bm{X}_t\le\bm{z}\}\right]
\le \int c(\bm{u})^Qf(\bm{u})d\bm{u}
<\infty.
\end{align*}

What is left to show, is the convergence of all finite dimensional distributions of $T_n$. To this end we will apply Cramér-Wold's device. Let $\lambda_1,\dots,\lambda_K\in\R\setminus\{0\}$ and consider 
\begin{align*}
\sum\limits_{j=1}^K \lambda_j T_n(s_j,\bm{z}_j)
=\sum\limits_{i=1}^{n}\xi_{n,i},
\end{align*} 
where $\xi_{n,i}:=\frac{1}{\sqrt{n}}U_i\sum_{j=1}^K \lambda_jI\{\bm{X}_i\le \bm{z}_j\}I\left\{\tfrac{i}{n}\le s_j\right\}$. 
Now  Corollary 1 in \cite{Rio199535} can be applied, which is a central limit theorem for strongly mixing triangular arrays. Following the notations in \cite{Rio199535} define $V_{n,l}:=Var(\sum_{i=1}^l\xi_{n,i})$ for all $l=1,\dots,n,$ and $n\in\N$. Let furthermore $Q_{n,i}$ be the c\`{a}dl\`{a}g inverse function of $t\mapsto P(|\xi_{n,i}|>t)$, i.e.
\[Q_{n,i}(u):=\sup\{t>0: P(|\xi_{n,i}|>t)>u\}, \ \forall \ u>0,\]
with the convention that $\sup \emptyset :=0$. Let $\{\tilde{\alpha}_n(t):t\in\N\}$ be the sequence of mixing coefficients of $\{\xi_{n,i}:1\le i\le n,n\in\N\}$. For $t\in(0,\infty)$ define $\tilde{\alpha}_n(t):=\tilde{\alpha}_n(\lf t \rf)$. Let its c\`{a}dl\`{a}g inverse function be defined by
\[\tilde{\alpha}_{n}^{-1}(u):=\sup\{t>0: \tilde{\alpha}_{n}(t)>u\}, \ \forall \ u>0.\]

Condition (a) in Corollary 1 in \cite{Rio199535} is easy to verify. 
%
Concerning condition (b) in aforementioned corollary
%
note that by Markov's inequality, it holds that for all $t>0$ and with $q:=Q\frac{2+\gamma}{2}$
\begin{align*}
P(|\xi_{n,i}|>t)
\le t^{-q}n^{-\frac{q}{2}}\tilde{M},
\end{align*}
where $\tilde{M}:=(\sum_{j=1}^{K}|\lambda_j|)^q M$ for $M<\infty$ from assumption \textbf{(U)}. 
%
Hence, $Q_{n,i}(u)\le u^{-\frac{1}{q}}n^{-\frac{1}{2}}\tilde{M}^{\frac{1}{q}}$ holds for all $u>0$. By similar arguments, we obtain $\tilde{\alpha}_n^{-1}(u)\le \tilde{A}-\log_a(u)$ for all $u>0$, where $\tilde{A}:=\log_a(A)$.
%
%
Furthermore, $V_{n,n}$ converges to $\sum_{j_1=1}^{K}\sum_{j_2=1}^{K}\lambda_{j_1}\lambda_{j_2}(s_{j_1}\mi s_{j_2}) \Sigma(\bm{z}_{j_1}\mi \bm{z}_{j_2})>0$. 
%
%
Putting the results together, it can be obtained that
\begin{align*}
&V_{n,n}^{-\frac{3}{2}}\sum\limits_{i=1}^{n}\int_{0}^{1}\tilde{\alpha}^{-1}_{n}\left(\tfrac{x}{2}\right) Q^2_{i,n}(x)\inf\left\{\tilde{\alpha}^{-1}_{n}\left(\tfrac{x}{2}\right)Q_{i,n}(x),\sqrt{V_{n,n}}\right\}dx\\
&\le\frac{1}{\sqrt{n}} \tilde{M}^{\frac{2}{q}} V_{n,n}^{-\frac{3}{2}}\int_{0}^{1}\left(\tilde{A}-\log_a\left(\tfrac{x}{2}\right)\right) x^{-\frac{2}{q}}\inf\left\{\left(\tilde{A}-\log_a\left(\tfrac{x}{2}\right)\right)x^{-\frac{1}{q}}\tilde{M}^{\frac{1}{q}},\sqrt{n}\sqrt{V_{n,n}}\right\}dx
\end{align*}
converges to zero, and thus condition (b) is satisfied. 
%

%
%
%
%
Applying Corollary 1 in \cite{Rio199535}, it holds that $\sum_{i=1}^{n}\xi_{n,i}/(V_{n,n})^{1/2}$ converges to the standard normal distribution and thus the assertion follows.
\end{proof}

\subsection{Proofs of lemmata}\label{Lemma-proof}

\begin{proof}[Proof of Lemma \ref{Lemma2}]
For some $l$-times differentiable function $h:\bm{J}_n\to \R$ define the norm
\[\|h\|_{l+\eta}:=\max\limits_{\substack{\bm{k}\in\N_0^d \\ 1\le |\bm{k}|\le l}}\sup\limits_{\bm{x}\in\bm{J}_n}\left|D^{\bm{k}}h(\bm{x})\right|+\max\limits_{\substack{\bm{k}\in\N_0^d \\ |\bm{k}|=l}}\sup\limits_{\substack{\bm{x},\bm{y}\in\bm{J}_n\\\bm{x}\neq \bm{y}}}\frac{\left|D^{\bm{k}}h(\bm{x})-D^{\bm{k}}h(\bm{y})\right|}{\|\bm{x}-\bm{y}\|^{\eta}}\]
and the function class $\cH:=\cC_{1,n}^{l+\eta}(\bm{J}_n):=\{h:\bm{J}_n\to\R: \|h\|_{l+\eta}\le 1,\sup_{\bm{x}\in\bm{J}_n}\left|h(\bm{x})\right|\le z_n\sqrt{\log n}\}$
with $z_n:=q_n\delta_n ((\log n)/(nh_n^d))^{1/2}$.
The third bandwidth condition in \textbf{(B2)} implies
\begin{align*}
\left(\sqrt{\frac{\log n}{nh_n^d}}+h_n^rp_n\right)q_n\delta_n=O\left(\sqrt{\frac{\log n}{nh_n^d}}q_n\delta_n\right)
\end{align*}
and thus Lemma \ref{Raten kernel} implies that
 $P(\hat{h}_n\in \cC_{1,n}^{l+\eta}(\bm{J}_n))\to 1$ as $n\to\infty$ holds for  $\hat{h}_n(\bm{x})=(m(\bm{x})-\hat{m}_n(\bm{x}))\om_n(\bm{x})$.
Let furthermore $\bm{X}_t\sim P$ and  $\cF:=\{\bm{x}\mapsto I\{\bm{x}\le \bm{z}\}:\bm{z}\in\R^d\}$. 
%
Then the assertion of the lemma follows if we show 
\begin{align*}
\sup\limits_{s\in[0,1]}\sup\limits_{\ph\in\cF}\sup\limits_{h\in\cH}\left|\frac{1}{\sqrt{n}}\sum\limits_{i=1}^{\lf ns \rf} \left(h(\bm{X}_i)\ph(\bm{X}_i)-\int h\ph dP\right)\right|=o_P(1).
\end{align*}

To this end let $\ep_{n1}=\ep_{n2}=n^{-1/2}$ and $\ep_{n3}=n^{-1/2}/(\log n)$ and let further $0=s_1<\dots<s_{K_n}=1$ partition $[0,1]$ in intervals of length $2\ep_{n1}$ such that $K_n=O(\ep_{n1}^{-1})$. Furthermore, let $J_n:=N_{[~]}\left(\ep_{n2},\cF,\|\cdot\|_{L_2(P)}\right)$ and $M_n:=N_{[~]}\left(\ep_{n3},\cH,\|\cdot\|_{\infty}\right)$, where $\|\cdot\|_{\infty}$ is the supremum norm on $\bm{J}_n$. Let $[\ph_1^{l},\ph_1^{u}],\dots,[\ph_{J_n}^{l},\ph_{J_n}^{u}]$ denote the brackets needed to cover $\cF$. Note that they can be chosen to be indicator functions and therefore non negative. Let furthermore $[h_1^{l},h_1^{u}],\dots,[h_{M_n}^{l},h_{M_n}^{u}]$ define the brackets needed to cover $\cH$. It can be shown that $J_n=O\left(\ep_{n2}^{-2d}\right)$ and $M_{n}=O(\exp(c_n^d\ep_{n3}^{-d/(l+\eta)}))$ and further
\begin{align*}
&\sup\limits_{s\in[0,1]}\sup\limits_{\ph\in\cF}\sup\limits_{h\in\cF}\left|\frac{1}{\sqrt{n}}\sum\limits_{i=1}^{\lf ns \rf}\left(h(\bm{X}_i)\ph(\bm{X}_i)-\int h\ph dP\right)\right|\\
&\le\max\limits_{\substack{\substack{1\le k\le K_n\\1\le j\le J_n}\\1\le m\le M_n}}\sup\limits_{\ph\in[\ph_j^{l},\ph_j^u]}\sup\limits_{h\in[h_m^l,h_m^u]}\left|\frac{1}{\sqrt{n}}\sum\limits_{i=1}^{\lf ns_k \rf}\left(h(\bm{X}_i)\ph(\bm{X}_i)-\int h\ph dP\right)\right|\\
&+\max\limits_{\substack{\substack{1\le k\le K_n\\1\le j\le J_n}\\1\le m\le M_n}}\sup\limits_{\substack{s\in[0,1]\\|s-s_k|\le\ep_{n1}}}\sup\limits_{\ph\in[\ph_j^{l},\ph_j^u]}\sup\limits_{h\in[h_m^l,h_m^u]}\Bigg|\frac{1}{\sqrt{n}}\sum\limits_{i=1}^{n}\left(h(\bm{X}_i)\ph(\bm{X}_i)-\int h\ph dP\right)\\
&~~~~~~~~~~~~~~~~~~~~~~~~~~~~~~~~~~~~~~~~~~~~~~~~~~~~~~\times\left(I\left\{\textstyle\frac{i}{n}\le s\right\}-I\left\{\textstyle\frac{i}{n}\le s_k\right\}\right)\Bigg|\\
%
&\le\max\limits_{\substack{\substack{1\le k\le K_n\\1\le j\le J_n}\\1\le m\le M_n}}\left\{\left|\frac{1}{\sqrt{n}}\sum\limits_{i=1}^{\lf ns_k\rf}\left(h_m^u(\bm{X}_i)I\{h_m^u(\bm{X}_i)\ge 0\}\ph_j^u(\bm{X}_i)-\int h_m^uI\{h_m^u\ge 0\}\ph_j^udP\right)\right|\right.\\
&\hspace{2cm}+\left|\frac{1}{\sqrt{n}}\sum\limits_{i=1}^{\lf ns_k\rf}\left(h_m^u(\bm{X}_i)I\{h_m^u(\bm{X}_i)< 0\}\ph_j^l(\bm{X}_i)-\int h_m^uI\{h_m^u< 0\}\ph_j^ldP\right)\right| \\
&\hspace{2cm}+\left|\frac{1}{\sqrt{n}}\sum\limits_{i=1}^{\lf ns_k\rf}\left(h_m^l(\bm{X}_i)I\{h_m^l(\bm{X}_i)\ge 0\}\ph_j^l(\bm{X}_i)-\int h_m^lI\{h_m^l\ge 0\}\ph_j^ldP\right)\right|\\
&\hspace{2cm}+\left.\left|\frac{1}{\sqrt{n}}\sum\limits_{i=1}^{\lf ns_k\rf}\left(h_m^l(\bm{X}_i)I\{h_m^l(\bm{X}_i)< 0\}\ph_j^u(\bm{X}_i)-\int h_m^lI\{h_m^l< 0\}\ph_j^udP\right)\right|\right\}\\
&+o(1).
\end{align*}
In what follows we only consider the first line on the right hand side, while the other ones can be treated similarly. 
%
%
We apply Theorem 2.1 of \cite{Liebscher199669} to the random variable (for $m,j,k$ fixed)
\[Z_{i}:=\left(h_m^u(\bm{X}_i)I\{h_m^u(\bm{X}_i)\ge 0\}\ph_j^u(\bm{X}_i)-\int h_m^uI\{h_m^u\ge 0\}\ph_j^udP\right)I\left\{\frac{i}{n}\le s_k\right\}. \]
The mixing coefficient of $\{Z_t:1\le t\le n\}$ can be bounded by the mixing coefficient of $\{\bm{X}_t:t\in\Z\}$ due to \cite{Bradley1985165}, Section 2, remark (iv). Further, the variables are centered and have a bound of order $O(z_n\log n)$. Applying Theorem 2.1 to $\sum_{i=1}^{n}Z_i$ yields for all $\epsilon>0$ and $n\in\N$ large enough
\begin{align*}
&P\left(\max\limits_{\substack{\substack{1\le k\le K_n\\1\le j\le J_n}\\1\le m\le M_n}}\left|\frac{1}{\sqrt{n}}\sum\limits_{i=1}^{\lf ns_k\rf}\left(h_m^u(\bm{X}_i)I\{h_m^u(\bm{X}_i)\ge 0\}\ph_j^u(\bm{X}_i)-\int h_m^uI\{h_m^u\ge 0\}\ph_j^udP\right)\right|>\epsilon\right)\\
&\le\hspace{-0.2cm}\sum\limits_{\substack{\substack{1\le k\le K_n\\1\le j\le J_n}\\1\le m\le M_n}}\hspace{-0.3cm}P\left(\left|\frac{1}{\sqrt{n}}\sum\limits_{i=1}^{\lf ns_k\rf}\left(h_m^u(\bm{X}_i)I\{h_m^u(\bm{X}_i)\ge 0\}\ph_j^u(\bm{X}_i)-\int h_m^uI\{h_m^u\ge 0\}\ph_j^udP\right)\right|>\epsilon\right)\\
&\le K_nJ_nM_n4\exp\left(-\frac{n\epsilon^2}{64n \lf (nh_n^d)^{1/2}\rf z_n^2\log(n)+\frac{8}{3}\sqrt{n}\epsilon \lf (nh_n^d)^{1/2}\rf z_n\log(n)^{1/2}}\right)\\
&+K_nJ_nM_n4\frac{n}{\lf (nh_n^d)^{1/2}\rf}\alpha\left(\lf (nh_n^d)^{1/2}\rf\right)\\
&=o(1),
\end{align*}
where the first and second bandwidth constraints in \textbf{(B2)} were used. Details are omitted for the sake of brevity. 
\end{proof}


\begin{proof}[Proof of Lemma \ref{Lemma:Teil mit m}]
First, using the uniform rates of convergence results in Lemma \ref{Raten kernel} applied to $(m(\bm{X}_t),\bm{X}_t)_{t\in\Z}$ together with the first and the last bandwidth condition in  assumption \textbf{(B2)} as well as the second condition in assumption \textbf{(B1)}, it can be shown that
\begin{align*}
\sup\limits_{\bm{z}\in\R^d}\left|\frac{1}{\sqrt{n}}\sum\limits_{i=1}^{n}\int_{\R^d}\left(m(\bm{y})-m(\bm{X}_i)\right)K_{h_n}(\bm{y}-\bm{X}_i)\om_n(\bm{y})I\{\bm{y}\le \bm{z}\}\left(\frac{f(\bm{y})}{\hat{f}_n(\bm{y})}-1\right)d\bm{y}\right|=o_P(1).
\end{align*}
Defining the function class
\[\cF_{n,1}:=\left\{\bm{x}\mapsto \int_{(-\bm{\infty},\bm{z}]}(m(\bm{y})-m(\bm{x}))K_{h_n}(\bm{y}-\bm{x})\om_n(\bm{y})d\bm{y}:\bm{z}\in\R^d\right\},\]
and imposing $\bm{X}_t\sim P$, the assertion of the lemma follows if we show
\begin{align}
\sup\limits_{\ph\in\cF_{n,1}}\left|\frac{1}{\sqrt{n}}\sum\limits_{i=1}^{n}\left(\ph(\bm{X}_i)-\int \ph dP\right)\right|=o_P(1),\label{eq:Teil mit m:step2}
\end{align}
\begin{align}
\sup\limits_{\ph\in\cF_{n,1}}\left|\int \ph dP\right|=o\left(n^{-1/2}\right).\label{eq:Teil mit m:step3}
\end{align} 
The proof of \eqref{eq:Teil mit m:step2} uses similar techniques to those in the proof of Lemma \ref{Lemma2}, while the proof of \eqref{eq:Teil mit m:step2} follows applying Taylor expansion for $m$ and $f$. Details are given in the supplement.
\end{proof}


\begin{proof}[Proof of Lemma \ref{Lemma:Teil mit U}]
First, using the uniform rates of convergence results in Lemma \ref{Raten kernel} applied to $(U_t,\bm{X}_t)_{t\in\Z}$ together with assumptions on the bandwidth, it can be shown that
\[\sup\limits_{\bm{z}\in\R^d}\left|\frac{1}{\sqrt{n}}\sum\limits_{i=1}^{n}U_i\int_{\R^d}K_{h_n}(\bm{y}-\bm{X}_i)\om_n(\bm{y})I\{\bm{y}\le \bm{z}\}\left(\frac{f(\bm{y})}{\hat{f}_n(\bm{y})}-1\right)d\bm{y}\right|=o_P(1).\]
Furthermore, it can be shown that uniformly in $\bm{z}\in\R^d$ and for $q:=Q\frac{2+\gamma}{2}>2$
\begin{align*}
&\frac{1}{\sqrt{n}}\sum\limits_{i=1}^{n}U_i\Big(\int_{(-\bm{\infty},\bm{z}]} K_{h_n}(\bm{y}-\bm{X}_i)\om_n(\bm{y})d\bm{y}-\om_n(\bm{X}_i)I\{\bm{X}_i\le\bm{z}\}\Big)\\
&=\frac{1}{\sqrt{n}}\sum\limits_{i=1}^{n}\left(U_iI\{|U_i|\le n^{1/q}\}\Big(\int_{(-\bm{\infty},\bm{z}]} K_{h_n}(\bm{y}-\bm{X}_i)\om_n(\bm{y})d\bm{y}-\om_n(\bm{X}_i)I\{\bm{X}_i\le\bm{z}\}\Big)\right.\\
&\hspace{1.5cm}\left.-E\Big[U_iI\{|U_i|\le n^{1/q}\}\Big(\int_{(-\bm{\infty},\bm{z}]} K_{h_n}(\bm{y}-\bm{X}_i)\om_n(\bm{y})d\bm{y}-\om_n(\bm{X}_i)I\{\bm{X}_i\le\bm{z}\}\Big)\Big]\right)\\
&+o_P(1).
\end{align*}
Defining the function class
\begin{align*}
&\cF_{n,2}:=\Big\{(u,\bm{x})\mapsto uI\{|u|\le n^{1/q}\}\Big(\int_{(-\bm{\infty},\bm{z}]}K_{h_n}(\bm{y}-\bm{x})\om_n(\bm{y})d\bm{y}-\om_n(\bm{x})I\{\bm{x}\le\bm{z}\}\Big):\bm{z}\in\R^d\Big\},
\end{align*}
and imposing $(U_t,\bm{X}_t)\sim P$, the assertion of the lemma follows if we show
\begin{align}
\sup\limits_{\ph\in\cF_{n,2}}\left|\frac{1}{\sqrt{n}}\sum\limits_{i=1}^{n}\left(\ph(U_i,\bm{X}_i)-\int \ph dP\right)\right|=o_P(1).\label{eq:Teil mit U_supp}
\end{align}
The proof of \eqref{eq:Teil mit U_supp} uses similar techniques as the proof of Lemma \ref{Lemma2}. Details are given in the supplement.
\end{proof}


\begin{proof}[Proof of Lemma \ref{Lemma1}]
It will be shown that uniformly in $s\in[0,1]$ and $\bm{z}\in\R^d$,
\begin{align}
\frac{1}{\sqrt{n}}\sum\limits_{i=1}^{\lf ns \rf}U_iI\{\bm{X}_i\le \bm{z}\}I\{\bm{X}_i\notin [-c_n,c_n]^d\}=o_P(1). \label{eq:Teil mit om}
\end{align}
To this end define the function class
\[\cF:=\left\{(u,\bm{x})\mapsto uI\{\bm{x}\le\bm{z}\}I\{\bm{x}\notin [-a,a]^d\}:\bm{z}\in\R^d, a\in\R_{+}\right\}\]
and for $s\in[0,1]$ and $\ph\in\cF$
\[G_n(s,\ph):=\frac{1}{\sqrt{n}}\sum\limits_{i=1}^{\lf ns \rf} \left(\ph(U_i,\bm{X}_i)-\int \ph dP\right),\]
where $(U_t,\bm{X}_t)\sim P$ and $\int \ph dP=0$. Similarly to the proof of Theorem \ref{decomposition} (ii) an application of Theorem 2.5 in \cite{Mohr2017} yields for all $\delta_n\searrow 0$ and with $d(\ph,\psi):=\left\|\ph-\psi\right\|_{L_{Q\frac{2+\gamma}{2}}(P)}$,
\begin{align}
\sup\limits_{\substack{\{s,t\in[0,1],\ph,\psi\in\cF: \\ |s-t|+d(\ph,\psi)<\delta_n\}}}\left|G_n(s,\ph)-G_n(t,\psi)\right|=o_P(1).\label{eq:Teil mit om:(A)}
\end{align}
Note that here the assumption $Q>(d+1)(2+\gamma)$ is needed as $\tilde{N}_{[~]}(\ep,\cF,\|\cdot\|_{L_2(P)})=O(\ep^{-2(d+1)})$. Then, for some fixed $\bm{z}\in\R^d$ defining $\ph_n(u,\bm{x}):=uI\{\bm{x}\le\bm{z}\}I\{\bm{x}\notin [-c_n,c_n]^d\}$, it holds that $\ph_n\in\cF$ for all $n\in\N$ and
\begin{align*}
d(\ph_n,0)
=\|\ph_n\|_{L_{Q\frac{2+\gamma}{2}}(P)}
\le\left(\int c(\bm{x})^QI\{\bm{x}\notin [-c_n,c_n]^d\}f(\bm{x})d\bm{x}\right)^{\frac{1}{Q}\frac{2}{2+\gamma}}
\overset{n\to\infty}{\longrightarrow} 0,
\end{align*}
where the convergence holds by the dominated convergence theorem as $\int c(\bm{x})^Qf(\bm{x})d\bm{x}<\infty$. With
\[\left(\int c(\bm{x})^QI\{\bm{x}\notin [-c_n,c_n]^d\}f(\bm{x})d\bm{x}\right)^{\frac{1}{Q}\frac{2}{2+\gamma}}=:\delta_n\searrow 0\]
it can therefore be concluded that
\begin{align*}
&\sup\limits_{s\in[0,1]}\sup\limits_{\bm{z}\in\R^d}\left|\frac{1}{\sqrt{n}}\sum\limits_{i=1}^{\lf ns \rf}U_iI\{\bm{X}_i\le \bm{z}\}I\{\bm{X}_i\notin [-c_n,c_n]^d\}\right|
\le \sup\limits_{\substack{\{s,t\in[0,1],\ph,\psi\in\cF: \\ |s-t|+d(\ph,\psi)<\delta_n\}}}\left|G_n(s,\ph)-G_n(t,\psi)\right|.
\end{align*}
With \eqref{eq:Teil mit om:(A)} the last term is $o_P(1)$ which proves the assertion in \eqref{eq:Teil mit om} and therefore the assertion of the lemma.
\end{proof}

\bigskip

\vspace{0.5cm}
\noindent {\large \textbf{Supplementary material}}

\section{Additional simulation results} 

In addition to the models in section \ref{simus} in the main paper, in this section we investigate a second order autoregression model of the following form
\begin{alignat*}{3}
&&\text{(model 4)} 
\hspace{1.8cm}&Y_t=m_t(Y_{t-1},Y_{t-2})+\si(Y_{t-1},Y_{t-2})\ep_t, \ \ep_t\sim\mathcal{N}(0,1), & \\
&~&&m_t(x_1,x_2)=\begin{cases}0.9x_1-0.4x_2, \ & t=1,\dots, \lf n/2\rf \\\left(0.9-\Delta_0\right)x_1-0.4x_2, \ & t=\lf n/2\rf +1,\dots, n\end{cases}, \hspace{1.2cm}
\end{alignat*}
with $\Delta_0\in\{0,1.3\}$. We consider three different choices for the conditional variance function, namely $\si^2(x_1,x_2)=1$ for an AR(2) model, $\si^2(x_1,x_2)=1+0.4x_1^2$ for an AR(2)-ARCH(1) model and $\si^2(x_1,x_2)=1+0.2x_1^2+0.2x_2^2$ for an AR(2)-ARCH(2) model. In all cases $H_0$ is fulfilled for $\Delta_0=0$ and $H_1$ for $\Delta_0\neq 0$ with a change in the regression function occurring in $\lf n/2\rf$.

The limiting distribution from Corollary \ref{cor:T_dach} applies, but as the covariate is multivariate we do not easily obtain an asymptotically distribution-free test. Thus  we apply the bootstrap procedure from section \ref{bootstrap}. Table \ref{sim_d=2} shows the rejection frequencies for all three models, when using the tests based on $T_{n1}^*$ and $T_{n2}^*$, as well as the bootstrap versions of $KS$ and $CM$ under both $H_0$ and $H_1$. It can be seen that under $H_0$ the tests reject a little more often than in the models considered in section \ref{simus}, overestimating the level of $5\%$ sometimes for finite sample sizes. Under the alternative the number of rejections increases rapidly for $T_{n1}^*$ and $T_{n2}^*$ with increasing $n$, while it stays relatively low for $KS$ and $CM$. In summary, the bootstrap tests perform reasonably well and are therefore an acceptable alternative to the tests using critical values of the limiting distribution, which are here not known due to multidimensional covariates. Furthermore in these models, they outperform the bootstrap versions of the CUSUM tests.

\begin{table}[!htbp]
\small
  \centering
	\caption{Rejection frequencies in model 4}\label{sim_d=2}
\begin{tabular}{@{}r r c r r r r c r r r r @{}} 
    \toprule
\multicolumn{2}{c}{}& \phantom{a} & \multicolumn{4}{c}{under $H_0$} & \phantom{a} & \multicolumn{4}{c}{under $H_1$}\\
\cmidrule{4-7} \cmidrule{9-12}
model & $n$  &  & $T_{n1}^*$   & $T_{n2}^*$  &  $KS$ & $CM$  &  & $T_{n1}^*$   & $T_{n2}^*$  &  $KS$ & $CM$\\ 
\midrule 
AR(2)    & $100$  & &  	$0.082$  &  $0.068$	 &  $0.082$  &  $0.074$ 	& &		$0.124$  &  $0.110$  &	$0.080$  &  $0.070$ \\
			   & $300$  & &   $0.064$  &  $0.070$  &  $0.054$  &  $0.048$   & &   $0.284$  &  $0.308$  &  $0.096$  &  $0.070$ \\		
				 & $500$  & &   $0.076$  &  $0.058$  &  $0.068$  &  $0.060$   & &   $0.480$  &  $0.532$  & $0.098$  &  $0.070$ \\
\midrule 
AR(2)-   & $100$  & &  	$0.076$  &  $0.060$  & 	$0.094$  &  $0.068$		&	&	  $0.098$  &  $0.106$  &  $0.070$  &  $0.058$ \\
ARCH(1)  & $300$  & &  	$0.084$  &  $0.098$  &	$0.086$  &  $0.096$  	&	&   $0.252$  &  $0.282$  & $0.088$  &  $0.074$ \\		
				 & $500$  & &  	$0.098$  &  $0.078$  & 	$0.080$  &  $0.074$		& &	  $0.476$  &  $0.484$  & 	$0.120$  &  $0.078$ \\		
\midrule 
AR(2)-   & $100$  & &  	$0.076$  &  $0.064$  &	$0.064$  &  $0.044$		&	&  	$0.096$  &  $0.104$  & $0.072$  &  $0.050$ \\
ARCH(2)  & $300$  & &  	$0.100$  &  $0.082$  &	$0.092$  &  $0.076$   &	&   $0.226$  &  $0.236$  &  $0.108$  &  $0.074$\\		
				 & $500$  & &  	$0.082$  &  $0.068$  &  $0.076$  &  $0.056$ 	&	&   $0.392$  &  $0.420$  &  $0.094$  &  $0.068$\\				
\bottomrule  
\end{tabular}%
\end{table}

\section{Proofs of auxiliary results} 

In this section we give the proofs of \eqref{eq:Teil mit m:step2} and \eqref{eq:Teil mit m:step3} in Lemma \ref{Lemma:Teil mit m} as well as \eqref{eq:Teil mit U_supp} in Lemma \ref{Lemma:Teil mit U}. 


\begin{proof}[Proof of \eqref{eq:Teil mit m:step2} and \eqref{eq:Teil mit m:step3} in Lemma \ref{Lemma:Teil mit m}]
For the proof of \eqref{eq:Teil mit m:step2} let $\ep_{n}:=n^{-1/2}/(\log n)$ and $J_n:=N_{[~]}\left(\ep_n,\cF_{n,1},\|\cdot\|_{L_1(P)}\right)$. It can be shown that there exists a partition $\bm{z}_1,\dots,\bm{z}_{J_n}$ of $\R^d$ such that $\|\ph_{j}^u-\ph_{j}^l\|_{L_1(P)}\le \ep_n$ for all $j\in\{1,\dots,J_n\}$, where  
\begin{align*}
&\ph_{j}^u(\bm{x}):=\hspace{-0.35cm}\int\limits_{(-\bm{\infty},\bm{z}_{j}]}\hspace{-0.25cm}(m(\bm{y})-m(\bm{x}))K_{h_n}\left(\bm{y}-\bm{x}\right)I\{(m(\bm{y})-m(\bm{x}))K_{h_n}\left(\bm{y}-\bm{x}\right)\ge 0\}\om_n(\bm{y})d\bm{y}\\
&+\int\limits_{(-\bm{\infty},\bm{z}_{j-1}]}(m(\bm{y})-m(\bm{x}))K_{h_n}\left(\bm{y}-\bm{x}\right)I\{(m(\bm{y})-m(\bm{x}))K_{h_n}\left(\bm{y}-\bm{x}\right)< 0\}\om_n(\bm{y})d\bm{y}
\end{align*}
and
\begin{align*}
&\ph_{j}^l(\bm{x}):=\hspace{-0.45cm}\int\limits_{(-\bm{\infty},\bm{z}_{j-1}]}\hspace{-0.35cm}(m(\bm{y})-m(\bm{x}))K_{h_n}\left(\bm{y}-\bm{x}\right)I\{(m(\bm{y})-m(\bm{x}))K_{h_n}\left(\bm{y}-\bm{x}\right)\ge 0\}\om_n(\bm{y})d\bm{y}\\
&+\int\limits_{(-\bm{\infty},\bm{z}_{j}]}(m(\bm{y})-m(\bm{x}))K_{h_n}\left(\bm{y}-\bm{x}\right)I\{(m(\bm{y})-m(\bm{x}))K_{h_n}\left(\bm{y}-\bm{x}\right)< 0\}\om_n(\bm{y})d\bm{y}.
\end{align*}
It then holds that $J_n=O(\ep_n^{-d})$. Using these brackets of $\cF_{n,1}$, it can be shown that
\begin{align*}
&\sup\limits_{\ph\in\cF_{n,1}}\left|\frac{1}{\sqrt{n}}\sum\limits_{i=1}^{n}\left(\ph(\bm{X}_i)-\int \ph dP\right)\right|\\
&=\max\limits_{1\le j\le J_n}\sup\limits_{\ph\in[\ph_{j}^l,\ph_{j}^u]}\left|\frac{1}{\sqrt{n}}\sum\limits_{i=1}^{n}\left(\ph(\bm{X}_i)-\int \ph dP\right)\right|\\
&\le \max\limits_{1\le j\le J_n} \max\left\{\left|\frac{1}{\sqrt{n}}\sum\limits_{i=1}^{n}\left(\ph_{j}^u(\bm{X}_i)-\int \ph_{j}^u dP\right)\right|,\left|\frac{1}{\sqrt{n}}\sum\limits_{i=1}^{n}\left(\ph_{j}^l(\bm{X}_i)-\int \ph_{j}^l dP\right)\right|\right\}+o(1),
\end{align*}
where it is sufficient to discuss the proof of
\[\max\limits_{1\le j\le J_n}\left|\frac{1}{\sqrt{n}}\sum\limits_{i=1}^{n}\left(\ph_{j}^u(\bm{X}_i)-\int \ph_{j}^u dP\right)\right|=o_P(1)\]
as the other assertion works analogously. By defining
\[\ph_{j,1}^u(\bm{x}):=\hspace{-0.15cm}\int\limits_{(-\bm{\infty},\bm{z}_{j}]}\hspace{-0.05cm}(m(\bm{y})-m(\bm{x}))K_{h_n}\left(\bm{y}-\bm{x}\right)I\{(m(\bm{y})-m(\bm{x}))K_{h_n}\left(\bm{y}-\bm{x}\right)\ge 0\}\om_n(\bm{y})d\bm{y}\]
and
\[\ph_{j,2}^u(\bm{x}):=\hspace{-0.35cm}\int\limits_{(-\bm{\infty},\bm{z}_{j-1}]}\hspace{-0.25cm}(m(\bm{y})-m(\bm{x}))K_{h_n}\left(\bm{y}-\bm{x}\right)I\{(m(\bm{y})-m(\bm{x}))K_{h_n}\left(\bm{y}-\bm{x}\right)< 0\}\om_n(\bm{y})d\bm{y}\]
it holds that $\ph_{j}^u(\bm{x})=\ph_{j,1}^u(\bm{x})+\ph_{j,2}^u(\bm{x})$. Thus, again the problem is reduced to showing
\[\max\limits_{1\le j\le J_n}\left|\frac{1}{\sqrt{n}}\sum\limits_{i=1}^{n}\left(\ph_{j,1}^u(\bm{X}_i)-\int \ph_{j,1}^u dP\right)\right|=o_P(1),\]
the other assertion works analogously. Similar to before we apply Theorem 2.1 of \cite{Liebscher199669} to the random variable $Z_i:=\ph_{j,1}^u(\bm{X}_i)-\int \ph_{j,1}^u dP$ for fixed $j$. Note that the mixing properties are the same as the ones of the original process. Further the variables are centered and posses a bound of order $O(h_nq_n)$. For all $\epsilon>0$ and $n\in\N$ large enough it can then be shown that
\begin{align*}
&P\left(\max\limits_{1\le j\le J_n}\left|\frac{1}{\sqrt{n}}\sum\limits_{i=1}^{n}\left(\ph_{j,1}^u(\bm{X}_i)-\int \ph_{j,1}^u dP\right)\right|>\epsilon\right)\\
&\le \sum\limits_{j=1}^{J_n}P\left(\left|\frac{1}{\sqrt{n}}\sum\limits_{i=1}^{n}\left(\ph_{j,1}^u(\bm{X}_i)-\int \ph_{j,1}^u dP\right)\right|>\epsilon\right)\\
&\le J_n 4\exp\left(-\frac{n\epsilon^2}{64n (\lf (\log n)^2\rf+1)h_n^2q_n^2+\frac{8}{3}\sqrt{n}\epsilon (\lf (\log n)^2\rf+1)h_nq_n}\right)\\
&+J_n4\frac{n}{ \lf (\log n)^2\rf+1}\alpha\left(\lf (\log n)^2\rf+1\right)\\
&=o(1),
\end{align*}
where eventually the last bandwidth condition in assumption \textbf{(B2)} was used. The assertion in \eqref{eq:Teil mit m:step3} can be shown by using Taylor's expansion for both $m$ and $f$ up to order $r-1$ and the assumptions in \textbf{(F1)}. Thus
\begin{align*}
\sup\limits_{\ph\in\cF_{n,1}}\left|\int \ph dP\right|
&=\sup\limits_{\bm{z}\in\R^d}\left|\int_{\R^d}\int_{\R^d}(m(\bm{y})-m(\bm{x}))\frac{1}{h_n^d}K\left(\frac{\bm{y}-\bm{x}}{h_n}\right)\om_n(\bm{y})I\{\bm{y}\le \bm{z}\}d\bm{y}f(\bm{x})d\bm{x}\right|\\
&\le\int_{\R^d}\left|\int_{\R^d}(m(\bm{y})-m(\bm{y}-\bm{t}h_n))K(\bm{t})\om_n(\bm{y})f(\bm{y}-\bm{t}h_n)d\bm{t}\right|d\bm{y}\\
&=O(h_n^rp_nq_n)=o\left(n^{-1/2}\right),
\end{align*}
where the last equality holds by the third condition in \textbf{(B2)}.
\end{proof}

\begin{proof}[Proof of \eqref{eq:Teil mit U_supp} in Lemma \ref{Lemma:Teil mit U}]
To this end let $\ep_{n}:=n^{-1/2}/(\log n)$ and $J_n:=N_{[~]}\left(\ep_n,\cF_{n,2},\|\cdot\|_{L_1(P)}\right)$. It can be shown that there exists a partition $\bm{z}_1,\dots,\bm{z}_{J_n}$ of $\R^d$ such that $\|\ph_{j}^u-\ph_{j}^l\|_{L_1(P)}\le \ep_n$ for all $j\in\{1,\dots,J_n\}$, where  
\begin{align*}
\ph_{j}^u(u,\bm{x})
&:=\ph_{j,1}^u(u,\bm{x})+\ph_{j,2}^u(u,\bm{x})+\ph_{j,3}^u(u,\bm{x}),
\end{align*}
where
\begin{align*}
&\ph_{j,1}^u(u,\bm{x}):=uI\{|u|\le n^{1/q}\}I\{u< 0\}\Bigg(\hspace{0.1cm}\int\limits_{(\bm{-\infty},\bm{z}_{j}]}\hspace{-0.3cm}K_{h_n}(\bm{y}-\bm{x})\om_n(\bm{y})d\bm{y}-\om_n(\bm{x})I\{\bm{x}\le \bm{z}_{j}\}\Bigg),\\
&\ph_{j,2}^u(u,\bm{x}):=uI\{|u|\le n^{1/q}\}I\{u\ge 0\}\Bigg(\hspace{0.05cm}\int\limits_{(\bm{-\infty},\bm{z}_{j-1}]}\hspace{-0.6cm}K_{h_n}(\bm{y}-\bm{x})\om_n(\bm{y})d\bm{y}-\om_n(\bm{x})I\{\bm{x}\le \bm{z}_{j-1}\}\Bigg),\\
&\ph_{j,3}^u(u,\bm{x}):=|u|I\{|u|\le n^{1/q}\} \Bigg(\int_{(-\bm{\infty},\bm{z}_{j}]}K_{h_n}(\bm{y}-\bm{x})I\{K_{h_n}(\bm{y}-\bm{x})\ge 0\}\om_n(\bm{y})d\bm{y}\\
&\hspace{5.5cm}-\int_{(-\bm{\infty},\bm{z}_{j-1}]}K_{h_n}(\bm{y}-\bm{x})I\{K_{h_n}(\bm{y}-\bm{x})\ge 0\}\om_n(\bm{y})d\bm{y}\Bigg),
\end{align*}
and similarly,
\begin{align*}
\ph_{j}^l(u,\bm{x})
&:=\ph_{j,1}^l(u,\bm{x})+\ph_{j,2}^l(u,\bm{x})+\ph_{j,3}^l(u,\bm{x}),
\end{align*}
where
\begin{align*}
&\ph_{j,1}^l(u,\bm{x}):=uI\{|u|\le n^{1/q}\}I\{u\ge 0\}\Bigg(\hspace{0.1cm}\int\limits_{(\bm{-\infty},\bm{z}_{j}]}\hspace{-0.3cm}K_{h_n}(\bm{y}-\bm{x})\om_n(\bm{y})d\bm{y}-\om_n(\bm{x})I\{\bm{x}\le \bm{z}_{j}\}\Bigg),\\
&\ph_{j,2}^l(u,\bm{x}):=uI\{|u|\le n^{1/q}\}I\{u< 0\}\Bigg(\hspace{0.05cm}\int\limits_{(\bm{-\infty},\bm{z}_{j-1}]}\hspace{-0.6cm}K_{h_n}(\bm{y}-\bm{x})\om_n(\bm{y})d\bm{y}-\om_n(\bm{x})I\{\bm{x}\le \bm{z}_{j-1}\}\Bigg),\\
&\ph_{j,3}^l(u,\bm{x}):=-|u|I\{|u|\le n^{1/q}\} \Bigg(\int_{(-\bm{\infty},\bm{z}_{j}]}K_{h_n}(\bm{y}-\bm{x})I\{K_{h_n}(\bm{y}-\bm{x})\ge 0\}\om_n(\bm{y})d\bm{y}\\
&\hspace{5.5cm}-\int_{(-\bm{\infty},\bm{z}_{j-1}]}K_{h_n}(\bm{y}-\bm{x})I\{K_{h_n}(\bm{y}-\bm{x})\ge 0\}\om_n(\bm{y})d\bm{y}\Bigg).
\end{align*}
It then holds that $J_n=O(\ep_n^-d)$. Using these brackets of $\cF_{n,2}$, it can be shown that
\begin{align*}
&\sup\limits_{\ph\in\cF_{n,2}}\left|\frac{1}{\sqrt{n}}\sum\limits_{i=1}^{n}\left(\ph(U_i,\bm{X}_i)-\int \ph dP\right)\right|\\
&\le\max\limits_{1\le j\le J_n}\left\{\left|\frac{1}{\sqrt{n}}\sum\limits_{i=1}^{n}\left(\ph_{j}^u(U_i,\bm{X}_i)-\int \ph_{j}^u dP\right)\right|+\left|\frac{1}{\sqrt{n}}\sum\limits_{i=1}^{n}\left(\ph_{j}^l(U_i,\bm{X}_i)-\int \ph_{j}^l dP\right)\right|\right\}+o(1).
\end{align*}
We will only consider the first term, as the second one is treated analogously. 
%
Similar to before we apply Theorem 2.1 of \cite{Liebscher199669} to the random variable $Z_i:=\ph_{j,1}^u(U_i,\bm{X}_i)-\int \ph_{j,1}^u dP$ for fixed $j$. Note that the mixing properties are the same as the ones of the original process. Further, the variables are centered and posses a bound of order $O(n^{1/q})$. For all $\epsilon>0$ and $n\in\N$ large enough it can then be shown that
\begin{align*}
&P\left(\max\limits_{1\le j\le J_n}\left|\frac{1}{\sqrt{n}}\sum\limits_{i=1}^{n}\left(\ph_{j,1}^u(U_i,\bm{X}_i)-\int \ph_{j,1}^u dP\right)\right|>\epsilon\right)\\
&\le \sum\limits_{j=1}^{J_n}P\left(\left|\frac{1}{\sqrt{n}}\sum\limits_{i=1}^{n}\left(\ph_{j,1}^u(U_i,\bm{X}_i)-\int \ph_{j,1}^u dP\right)\right|>\epsilon\right)\\
&\le J_n 4\exp\left(-\frac{n\epsilon^2}{64n (\lf \log(n)^2\rf+1)h_n+\frac{8}{3}\sqrt{n}\epsilon (\lf \log(n)^2\rf+1)n^{1/q}}\right)\\
&+J_n4\frac{n}{ \lf \log(n)^2\rf+1}\alpha\left(\lf \log(n)^2\rf+1\right)\\
&=o(1),
\end{align*}
due to the third bandwidth condition in assumption \textbf{(B2)} and because $q>2$.
\end{proof}


\bibliographystyle{apa}
\bibliography{mybibfile}

\begin{thebibliography}{}

\bibitem[\protect\astroncite{Bradley}{1985}]{Bradley1985165}
Bradley, R.~C. (1985).
\newblock {Basic properties of strong mixing conditions}.
\newblock In Eberlein, E. and Taqqu, M.~S., editors, {\em Dependence in
  Probability and Statistics}, pages 165--192. Birkh{\"{a}}user, Boston.

\bibitem[\protect\astroncite{Burke and Bewa}{2013}]{Burke2013261}
Burke, M.~D. and Bewa, G. (2013).
\newblock {Change-Point Detection for General Nonparametric Regression Models}.
\newblock {\em Open Journal of Statistics}, 3:261--267.

\bibitem[\protect\astroncite{Cai}{2007}]{Cai2007163}
Cai, Z. (2007).
\newblock {Trending time-varying coefficient time series models with serially
  correlated errors}.
\newblock {\em J. Econom.}, 136:163--188.

\bibitem[\protect\astroncite{Delgado and Manteiga}{2001}]{Delgado20011469}
Delgado, M.~A. and Manteiga, W.~G. (2001).
\newblock Significance testing in nonparametric regression based on the
  bootstrap.
\newblock {\em Ann. Statist.}, 29:1469--1507.

\bibitem[\protect\astroncite{Hafner and Herwartz}{2000}]{Hafner2000177}
Hafner, C.~M. and Herwartz, H. (2000).
\newblock {Testing for linear autoregressive dynamics under
  heteroskedasticity}.
\newblock {\em Econom. J.}, 3:177--197.

\bibitem[\protect\astroncite{Hansen}{2008}]{Hansen2008726}
Hansen, B.~E. (2008).
\newblock {Uniform convergence rates for kernel estimation with dependent
  data}.
\newblock {\em Econom. Theory}, 24:726--748.

\bibitem[\protect\astroncite{Hidalgo}{1995}]{Hidalgo1995671}
Hidalgo, J. (1995).
\newblock {A Nonparametric Conditional Moment Test for Structural Stability}.
\newblock {\em Econom. Theory}, 1:671--698.

\bibitem[\protect\astroncite{Honda}{1997}]{Honda199745}
Honda, T. (1997).
\newblock {The CUSUM tests with nonparametric regression residuals}.
\newblock {\em J. Japan Statist. Soc.}, 27:45--63.

\bibitem[\protect\astroncite{Hu\v{s}kov\'{a} and Antoch}{2003}]{Huskova2003201}
Hu\v{s}kov\'{a}, M. and Antoch, J. (2003).
\newblock {Detection of structural changes in regression}.
\newblock {\em Tatra Mt. Math. Publ.}, 26:201--215.

\bibitem[\protect\astroncite{Koul and Stute}{1999}]{Koul1999204}
Koul, H.~L. and Stute, W. (1999).
\newblock Nonparametric model checks for time series.
\newblock {\em Ann. Statist.}, 27:204--236.

\bibitem[\protect\astroncite{Krei{\ss}}{1997}]{Kreiss1997}
Krei{\ss}, J.-P. (1997).
\newblock {Asymptotic properties of residual bootstrap for autoregressions}.
\newblock Technical report, Technical University of Braunschweig, Germany.
\newblock
  \url{https://www.tu-braunschweig.de/Medien-DB/stochastik/kreiss-1997.pdf}.

\bibitem[\protect\astroncite{Krei{\ss} and Lahiri}{2012}]{Kreiss20123}
Krei{\ss}, J.-P. and Lahiri, S.~N. (2012).
\newblock {Bootstrap Methods for Time Series}.
\newblock In T.~Subba~Rao, S. S.~R. and Rao, C., editors, {\em Handbook of
  Statistics}, pages 3--26. Elsevier, Amsterdam.

\bibitem[\protect\astroncite{Kristensen}{2009}]{Kristensen20091433}
Kristensen, D. (2009).
\newblock {Uniform convergence rates of kernel estimators with heterogeneous
  dependent data}.
\newblock {\em Econom. Theory}, 25:1433--1445.

\bibitem[\protect\astroncite{Kristensen}{2012}]{Kristensen2012420}
Kristensen, D. (2012).
\newblock {Non-parametric detection and estimation of structural change}.
\newblock {\em Econom. J.}, 15:420--461.

\bibitem[\protect\astroncite{Liebscher}{1996}]{Liebscher199669}
Liebscher, E. (1996).
\newblock {Strong convergence of sums of $\alpha$-mixing random variables with
  applications to density estimation}.
\newblock {\em Stochastic Process. Appl.}, 65:69--80.

\bibitem[\protect\astroncite{Liu}{1988}]{Liu19881696}
Liu, R.~Y. (1988).
\newblock {Bootstrap procedure under some non-i.i.d. models}.
\newblock {\em Ann. Statist.}, 16:1696--1708.

\bibitem[\protect\astroncite{Mohr}{2018a}]{Mohr2018}
Mohr, M. (2018a).
\newblock {\em {Changepoint detection in a nonparametric time series regression
  model}}.
\newblock {PhD} thesis, University of Hamburg.
\newblock \url{http://ediss.sub.uni-hamburg.de/volltexte/2018/9416/}.

\bibitem[\protect\astroncite{Mohr}{2018b}]{Mohr2017}
Mohr, M. (2018b).
\newblock {Weak Convergence of Sequential Empirical Processes under Weak
  Dependence}.
\newblock preprint on arxiv at \url{https://arxiv.org/abs/1711.05112}.

\bibitem[\protect\astroncite{Perron and Zhou}{2008}]{Perron2008}
Perron, P. and Zhou, J. (2008).
\newblock {Testing Jointly for Structural Changes in the Error Variance and
  Coefficients of a Linear Regression Model}.
\newblock {\em Working paper, Boston University}.

\bibitem[\protect\astroncite{Pitarakis}{2004}]{Pitarakis200432}
Pitarakis, J.-Y. (2004).
\newblock {Least squares estimation and tests of breaks in mean and variance
  under misspecification}.
\newblock {\em Econom. J.}, 7:32--54.

\bibitem[\protect\astroncite{Rio}{1995}]{Rio199535}
Rio, E. (1995).
\newblock {About the Lindeberg method for strongly mixing sequences.}
\newblock {\em ESAIM Probab. Statist.}, 1:35--61.

\bibitem[\protect\astroncite{Shao}{2010}]{Shao2010218}
Shao, X. (2010).
\newblock {The Dependent Wild Bootstrap}.
\newblock {\em J. Am. Stat. Assoc.}, 105:218--235.

\bibitem[\protect\astroncite{Stute}{1997}]{Stute1997613}
Stute, W. (1997).
\newblock {Nonparametric model checks for regression}.
\newblock {\em Ann. Statist.}, 25:613--641.

\bibitem[\protect\astroncite{Su and White}{2010}]{Su20101761}
Su, L. and White, H. (2010).
\newblock {Testing structural change in partially linear models}.
\newblock {\em Econom. Theory}, 26:1761--1806.

\bibitem[\protect\astroncite{Su and Xiao}{2008}]{Su2008347}
Su, L. and Xiao, Z. (2008).
\newblock {Testing structural change in time-series nonparametric regression
  models}.
\newblock {\em Stat. Interface}, 1:347--366.

\bibitem[\protect\astroncite{van~der Vaart and Wellner}{1996}]{vanderVaart1996}
van~der Vaart, A.~W. and Wellner, J.~A. (1996).
\newblock {\em {Weak convergence and empirical processes}}.
\newblock Springer, New York.

\bibitem[\protect\astroncite{Vogt}{2015}]{Vogt2015811}
Vogt, M. (2015).
\newblock {Testing for structural change in time-varying nonparametric
  regression models}.
\newblock {\em Econom. Theory}, 31:811--859.

\bibitem[\protect\astroncite{Wu}{1986}]{Wu19861261}
Wu, C. F.~J. (1986).
\newblock {Jackknife, bootstrap and other resampling methods in regression
  analysis}.
\newblock {\em Ann. Statist.}, 14:1261--1295.

\bibitem[\protect\astroncite{Wu}{2016}]{Wu2016151}
Wu, J. (2016).
\newblock {Detecting structural changes under nonstationary volatility}.
\newblock {\em Econom. Lett.}, 146:151--154.

\bibitem[\protect\astroncite{Xu}{2015}]{Xu2015274}
Xu, K.-L. (2015).
\newblock Testing for structural change under non-stationary variances.
\newblock {\em Econom. J.}, 18:274--305.

\end{thebibliography}

\end{document}